\documentclass[reqno]{amsart}
\usepackage{amsmath, amsthm, amssymb, amstext}
\usepackage{hyperref,xcolor}
\hypersetup{
	pdfborder={0 0 0},
	colorlinks,
}
\usepackage{enumitem}
\allowdisplaybreaks

\newtheorem{theorem}{Theorem}
\newtheorem{remark}[theorem]{Remark}
\newtheorem{lemma}[theorem]{Lemma}
\newtheorem{proposition}[theorem]{Proposition}

\newtheorem{definition}[theorem]{Definition}

\newcommand{\Ss}{\textnormal{(S}_+\textnormal{)}}

\newcommand{\N}{\mathbb{N}}

\newcommand{\R}{\mathbb{R}}

\newcommand*\diff{\mathrm{d}}

\newcommand{\Lp}[1]{L^{#1}(\Omega)}

\newcommand{\Wp}[1]{W^{1,#1}(\Omega)}

\newcommand{\Wpzero}[1]{W^{1,#1}_0(\Omega)}

\newcommand{\eps}{\varepsilon}

\newcommand{\into}{\int_{\Omega}}

\newcommand{\weak}{\rightharpoonup}

\newcommand{\Linf}{L^{\infty}(\Omega)}
\newcommand{\close}{\overline{\Omega}}

\renewcommand{\l}{\left}
\renewcommand{\r}{\right}

\numberwithin{theorem}{section}
\numberwithin{equation}{section}

\title[Variable exponent multi phase problems]{Regularity and uniqueness to multi-phase problem with variable exponent}\thanks{Research supported by NNSF of China (No. 12371110).}

\author[G.\,Dai]{Guowei Dai}
\address[G.\,Dai]{School of Mathematical Sciences, Dalian University of Technology, Dalian, 116024, PR China}
\email{daiguowei@dlut.edu.cn}

\author[F.\,Vetro]{Francesca Vetro}
\address[F.\,Vetro]{Independent Researcher, 90100 Palermo, Italy}
\email{francescavetro80@gmail.com}


\begin{document}

\begin{abstract}
In this paper, we consider a new class of multi phase operators with variable exponents, which reflects the inhomogeneous characteristics of hardness changes when multiple different materials are combined together. We at first deal with the corresponding functional
spaces, namely the Musielak-Orlicz
Sobolev spaces, hence we investigate their regularity
properties and the extension of the classical Sobolev embedding results to the new context.
Then, we focus on the regularity properties of our operators, and prove that these operators are bounded, continuous, strictly monotone, coercive and satisfy the $\left(S_+\right)$-property. Further, we discuss suitable
problems driven by such operators. In particular,
we deal with Dirichlet problems in which the nonlinearity is
gradient dependent. Under very general assumptions, we establish the existence of a nontrivial solution for such problems. Also, we give additional conditions on the nonlinearity which guarantee the uniqueness of solution. Lastly, we produce some local regularity results (namely, Cacciop\-poli-type inequality, Sobolev-Poincar\'{e}-type inequalities and higher integrability) for minimizers of the integral functionals corresponding to the operators.
\end{abstract}

\maketitle

\section{Introduction}\label{sec1}
Let $\Omega \subseteq \mathbb{R}^N\, (N \geq 2)$ be a bounded domain with Lipschitz boundary. In this paper, we focus on a new class of operators with variable exponents. Precisely, we deal with operators of the form
\begin{align} \label{operator}
	\text{div} \, \l(|\nabla u|^{p(x)-2}\nabla u +\mu_1(x)|\nabla u|^{q(x)-2}\nabla u + \mu_2(x)|\nabla u|^{r(x)-2}\nabla u \r)
\end{align}
where $p(\cdot), q(\cdot), r(\cdot) \in C(\close)$ are such that $1 < p(x) <N$, \,$ p(x) <q(x) < r(x)$ for all $x \in \close$ and further $0 \leq \mu_1, \mu_2 \in \Linf$.
To the best of our knowledge, this is the first work dealing with problems driven by the operator with variable exponents given in the general form \eqref{operator}. Also, we emphasize that the integral functionals of type
	\begin{align} \label{funmulti}
	\omega \mapsto \int_\Omega \big(  |\nabla  \omega|^{p(x)} + \mu_1(x)|\nabla  \omega|^{q(x)} + \mu_2(x)|\nabla  \omega|^{r(x)}\big)\,\diff x
\end{align}
have a strongly non-uniform ellipticity and discontinuity of growth order at points where $\mu_i = 0$. Hence, they are
a valuable tool in order to describe
the behaviour of inhomogeneous materials whose strengthening properties, connected to the exponent dominating the growth of the gradient variable, significantly change with the point. Moreover, the functionals given in \eqref{funmulti} have two modulating coefficients $\mu_1(\cdot), \mu_2(\cdot)$ which are able
to dictate the geometry of composites made of three different materials.

Such operator arises as natural generalization both of the variable exponent double phase operator
\begin{align} \label{operator2}
	\text{div} \, \l(|\nabla u|^{p(x)-2}\nabla u +\mu_1(x)|\nabla u|^{q(x)-2}\nabla u \r)
\end{align}
studied by Crespo-Blanco-Gasinski-Harjulehto-Winkert in recent paper \cite{Crespo-Blanco-Gasinski-Harjulehto-Winkert-2022}, and of multi phase operator with constant exponents considered by De Filippis-Oh in \cite{DeFilippis-Oh-2019}. We in fact stress that if $\mu_2 = 0$ then the operator in \eqref{operator} becomes the variable exponent double phase operator given in \eqref{operator2}. Also, if $p, q, r$ are constants then the operator in \eqref{operator} reduces to the multi phase operator with
constant exponents.

Here inspired by \cite{Crespo-Blanco-Gasinski-Harjulehto-Winkert-2022} we start our study by focusing on the functional space related to the operator introduced in \eqref{operator} (see Section \ref{sec3}). We show that such space is a reflexive Banach space (see Proposition \ref{prop1}) and further we check that the classical Sobolev embedding results extend to it (see Proposition \ref{embedding}).
Then, we investigate the regularity properties of the operator given in \eqref{operator} (see Section \ref{sec4}). In particular, we show that it is bounded (that is, it maps bounded sets into bounded sets), continuous, strictly monotone, coercive and in addition it satisfies the $\left(S_+\right)$-property.

Next, we also discuss certain
problems driven by our new operator. We in particular focus on the following Dirichlet problem
\begin{equation} \label{problem}
	\begin{aligned}
		-\operatorname{div}\left( |\nabla u|^{p(x)-2}\nabla u + \mu_1(x)|\nabla u|^{q(x)-2}\nabla u + \mu_2(x)|\nabla u|^{r(x)-2}\nabla u\right)\\
		= f(x,u, \nabla u) && \text{in } \Omega,\\
		\hspace{-30ex}u  = 0 &&\text{on } \partial\Omega
	\end{aligned}
\end{equation}
where we suppose $r(x) < p^*(x):= N p(x)/(N-p(x))$ (the critical exponent corresponding to $p(x)$)
and where $f\colon \Omega \times \R \times \R^N \to\R$ is a Carath\'eodory function which depends on the gradient of the solution. We point out that the appearance of a nonlinearity that depends on the gradient of the solution is a source of difficulty for the analysis of problem \eqref{problem}. Indeed, it gives to our problem a non-variational structure and so we cannot apply directly the standard variational tools. Consequently, here we use a topological approach based on the surjectivity of pseudomonotone operators. In this way, under suitable growth assumptions (see hypothesis \eqref{H2}), we are able to establish the existence of at least one
nontrivial weak solution for problem \eqref{problem} (see Theorem \ref{existence}). Furthermore, under additional hypothesis on the nonlinearity $f$ (see hypothesis \eqref{H3}), we have also the uniqueness of the solution when $p(x) = 2$ for all $x \in \close$ (see Theorem \ref{unicity1}), when $q(x) = 2$ for all $x \in \close$ and $\inf_{x \in \close} \mu_1(x) > 0$ (see Theorem \ref{unicity2}), and when $r(x) = 2$ for all $x \in \close$ and $\inf_{x \in \close} \mu_2(x) > 0$ (see Theorem \ref{unicity3}). We emphasize that our existence and uniqueness results, given in Theorem \ref{existence} and Theorem \ref{unicity1}, respectively, extend the ones of Crespo-Blanco-Gasinski-Harjulehto-Winkert \cite[Theorems 4.4 and 4.5]{Crespo-Blanco-Gasinski-Harjulehto-Winkert-2022} to the case of three variable exponents.
	
We note that functionals of type
	\begin{align} \label{fundouble}
		\omega \mapsto \int_\Omega \big(  |\nabla  \omega|^p + \mu(x)|\nabla  \omega|^q\big)\,\diff x, \quad 1< p < q < N,
	\end{align}
were first considered by Marcellini \cite{Marcellini-1991} and Zhikov \cite{Zhikov-1986} in the context of problems of nonlinear elasticity.
But such functionals have also applications, for example, in the study of duality theory
and of the Lavrentiev gap phenomenon, see Zhikov \cite{Zhikov-1995,Zhikov-2011} and Papageorgiou-R\u{a}dulescu-Repov\v{s} \cite{Papageorgiou-Radulescu-Repovs-2019-2}, and in the context of problems of the calculus of variations,
see Marcellini \cite{Marcellini-1989, Marcellini-1991}. We in addition mention that regularity results for local minimizers of such functionals were produced by Baroni-Colombo-Mingione \cite{Baroni-Colombo-Mingione-2018} and
Colombo-Mingione \cite{Colombo-Mingione-2015, Colombo-Mingione-2015-1}.

Among of papers involving the variable exponent double phase operator, we recall the recent work of Aberqi-Bennouna-Benslimane-Ragusa \cite{Aberqi-Bennouna-Benslimane-Ragusa-2022} for existence results in complete manifolds, Albalawi-Alharthi-Vetro \cite{Albalawi-Alharthi-Vetro-2022} for $(p(\cdot),q(\cdot))$-Laplace type problems with convection, Bahrouni-R\u{a}dulescu-Winkert \cite{Bahrouni-Radulescu-Winkert-2020} for double phase problems of Baou\-endi-Grushin type operator, Leonardi-Papageorgiou \cite{Leonardoi-Papageorgiou-2022} for concave-convex problems, Vetro-Winkert \cite{Vetro-Winkert-2023} for parametric problems involving superlinear nonlinearities, Zeng-R\u{a}dulescu-Winkert \cite{Zeng-Radulescu-Winkert-2022} for multivalued problems and Zhang-R\u{a}dulescu \cite{Zhang-Radulescu-2018} for the existence of multiple solutions for double phase anisotropic variational problems, see in addition the references therein.

Also, we refer to the papers of Colasuonno-Squassina \cite{Colasuonno-Squassina-2016} for eigenvalue problems of double phase type, Gasi\'{n}ski-Papageorgiou \cite{Gasinski-Papageorgiou-2021} for locally Lipschitz right-hand sides, Gasi\'{n}ski-Winkert \cite{Gasinski-Winkert-2020b, Gasinski-Winkert-2020} for convection problems and constant sign-solutions, Papageorgiou-Vetro \cite{Papageorgiou-Vetro-2019} for superlinear problems, Papageorgiou-Vetro-Vetro \cite{Papageorgiou-Vetro-Vetro-2021} for parametric Robin problems, Perera-Squassina \cite{Perera-Squassina-2018} for Morse theoretical approach, Vetro-Winkert \cite{Vetro-Winkert-2022} for parametric convective problems. Finally, we mention the recent paper of Liu-Dai \cite{Liu-Dai-2018} where existence and multiplicity results
including
sign-changing solutions are stated for double
phase problems in the constant case. We in addition stress that Liu-Dai \cite{Liu-Dai-2018} derive interesting regularity properties of double phase operator with constant exponents.
	
So we devote the final section of this paper to the study of the minimizers of multi phase integral functionals of type \eqref{funmulti} where $p(\cdot), q(\cdot), r(\cdot)$ are now assumed to be H\"{o}lder continuous (see Section \ref{sec6}). In particular, we here produce some local regularity results. We mention that regularity results in the case of $p, q, r$ constants were furnished by De Filippis-Oh in their recent paper \cite{DeFilippis-Oh-2019}. Also, we remark that Ragusa-Tachikawa \cite{Ragusa-Tachikawa-2020} stated a regularity theorem for minimizers of integral functionals of type \eqref{funmulti} in the case when $\mu_2 = 0$.

We prove, under very general assumptions (see hypothesis \eqref{I}), first a Cacciop\-poli-type inequality (see Theorem \ref{cacioppoli}) and then we give Sobolev-Poincar\'{e}-type inequalities (see Theorems \ref{SobPonIne} and \ref{SobPoiInEE}). We stress that all these inequalities are instrumental for the study of the higher integrability property of the gradient of the minimizers of our integral functionals (see Theorems \ref{HI} and \ref{HI2}).

\section{Mathematical background}\label{sec2}
This section is devoted to introduce the function  framework space we will deal with. Thus, we recall some basic facts from the theory of variable exponent Sobolev spaces and of Musielak-Orlicz Sobolev spaces. Such issues can be enlarged referring to the books of Diening-Harjulehto-H\"{a}st\"{o}-R$\mathring{\text{u}}$\v{z}i\v{c}ka \cite{Diening-Harjulehto-Hasto-Ruzicka-2011}, Harjulehto-Hasto \cite{Harjulehto-Hasto-2019}, Motreanu-Motreanu-Papageorgiou \cite{Motreanu-Motreanu-Papageorgiou-2014} and Musielak \cite{Musielak-1983}.
We also mention the recent paper of Crespo-Blanco-Gasinski-Harjulehto-Winkert \cite{Crespo-Blanco-Gasinski-Harjulehto-Winkert-2022}.

Let $\Omega \subseteq \mathbb{R}^N$ be a bounded domain with Lipschitz boundary $\partial \Omega$. Given $m \in C(\close)$, with $m(x) > 1$ for all $x \in \close$, we set
\begin{align*}
	m^-=\min_{x\in \close}m(x)
	\quad\text{and}\quad
	m^+=\max_{x\in\close} m(x).
\end{align*}
Also, we write $m^{\prime}$ by the conjugate variable exponent to $m$, that is,
\begin{align*}
	\frac{1}{m(x)} + \frac{1}{m^{\prime}(x)} = 1 \quad \text{for all} \ x \in \close.
\end{align*}
Let $M(\Omega)$ be the set of all measurable  functions $u\colon \Omega\to\R$. As usual, we identify two such functions which
differ only on a Lebesgue-null set. We denote by $L^{m(\cdot)}(\Omega)$ the variable exponent Lebesgue space, that is,
\begin{align*}
	L^{m(\cdot)}(\Omega) = \left\{u\in M(\Omega)\, :\, \int_\Omega |u|^{m(x)}\,\diff x < +\infty\right\}
\end{align*}
equipped with the Luxemburg norm
\begin{align*}
	\|u\|_{{m(\cdot)}}= \inf  \left\{ \alpha > 0 \,:\,  \int_\Omega \l| \dfrac{u}{\alpha} \r|^{m(x)}\diff x  \leq 1 \right\}.
\end{align*}
Then, the corresponding modular function $\rho_{m(\cdot)}: \Lp{m(\cdot)} \to \R$ is given by
\begin{align*}
	\rho_{m(\cdot)}(u):=\int_\Omega |u|^{m(x)}\,\diff x \quad \text{for all} \ u \in \Lp{m(\cdot)}.
\end{align*}
This modular function is closely related to the Luxemburg norm. In fact, we have the following result.

\begin{proposition}\label{Prop0}
	Let $m \in C(\close)$ be such that $m(x) > 1$ for all $x \in \close$.
	Then the following hold:
	\begin{enumerate}
		\item[\textnormal{(i)}]
		if $u \neq 0$, we have $\|u\|_{m(\cdot)} = \lambda $ if and only if $\rho_{m(\cdot)}(\frac{u}{\lambda}) = 1$;
		\item[\textnormal{(ii)}]
		$\|u\|_{m(\cdot)} < 1$ (resp. $> 1, = 1$) if and only if $\rho_{m(\cdot)}(u) < 1$ (resp. $> 1, = 1$);
		\item[\textnormal{(iii)}]
		if $\|u\|_{m(\cdot)} < 1$ then $\|u\|_{m(\cdot)}^{m^+} \leq \rho_{m(\cdot)} (u) \leq \|u\|_{m(\cdot)}^{m^-}$;
		\item[\textnormal{(iv)}]
		if $\|u\|_{m(\cdot)} > 1$ then $\|u\|_{m(\cdot)}^{m^-} \leq \rho_{m(\cdot)} (u) \leq \|u\|_{m(\cdot)}^{m^+}$;
		\item[\textnormal{(v)}]
		$\|u\|_{m(\cdot)} \rightarrow 0$ if and only if $\rho_{m(\cdot)} (u) \rightarrow 0$;
		\item[\textnormal{(vi)}]
		$\|u\|_{m(\cdot)} \rightarrow + \infty$ if and only if $\rho_{m(\cdot)} (u) \rightarrow + \infty$;
		\item[\textnormal{(vii)}]
		$\|u\|_{m(\cdot)} \rightarrow 1$ if and only if $\rho_{m(\cdot)} (u) \rightarrow 1$;
		\item[\textnormal{(viii)}]
		if $u_n \rightarrow u$ in $\Lp{m(\cdot)}$ we have $\rho_{m(\cdot)} (u_n) \rightarrow \rho_{m(\cdot)} (u)$.
	\end{enumerate}
\end{proposition}

We recall that $L^{m(\cdot)}(\Omega)$ is a separable, uniformly convex and hence reflexive Banach space with dual space $L^{m(\cdot)}(\Omega)^* = L^{m^{\prime}(\cdot)}(\Omega)$.
Further, for all $u \in L^{m(\cdot)}(\Omega) $ and for all $v \in L^{m^{\prime}(\cdot)}(\Omega)$, the H\"{o}lder type inequality
\begin{align*}
	\int_\Omega |u v| \, \diff x \leq \l[ \frac{1}{m^-} + \frac{1}{(m^{\prime})^-}\r] \|u\|_{m} \|v\|_{m^{\prime}} \leq 2 \|u\|_{m} \|v\|_{m^{\prime}}
\end{align*}
holds true. We also recall that for given $m_1, m_2 \in C(\close)$ with $1 < m_1(x) \leq m_2(x)$ for all $x \in \close$, we have the continuous embedding of $\Lp{m_2(\cdot)}$ in $\Lp{m_1(\cdot)}$.

Using the Lebesgue space $L^{m(\cdot)}(\Omega)$, we define the corresponding Sobolev space with variable exponent $W^{1,m(\cdot)}(\Omega)$ by
\begin{align*}
W^{1,m(\cdot)}(\Omega) = \left\{u\in L^{m(\cdot)}(\Omega)\, :\, |\nabla u| \in L^{m(\cdot)}(\Omega)\right\}
\end{align*}
equipped with the norm
\begin{align*}
	\|u\|_{1,m(\cdot)} = \|u\|_{m(\cdot)} + \|\nabla u\|_{m(\cdot)},
\end{align*}
where $\|\nabla u\|_{m(\cdot)} = \|\,|\nabla u|\, \|_{m(\cdot)}$. In addition, we denote by $W^{1,m(\cdot)}_0(\Omega)$ the completion of $C^\infty_0(\Omega)$ in $W^{1,m(\cdot)}(\Omega)$.

We point out that both the spaces $W^{1,m(\cdot)}(\Omega)$ and $W^{1,m(\cdot)}_0(\Omega)$
are uniformly convex, separable and reflexive Banach spaces. Moreover, the Poincar\'{e} inequality is valid for $W^{1,m(\cdot)}_0(\Omega)$, that is, there exists $c > 0$ such that
\begin{align*}
	\|u\|_{m(\cdot)} \leq c \, \| \nabla u\|_{m(\cdot)}
\end{align*}
for all $u \in W^{1,m(\cdot)}_0(\Omega)$.

According to this, on $W^{1,m(\cdot)}_0(\Omega)$ we can use the equivalent norm
\begin{align*}
	\|u\|_{1, m(\cdot), 0} := \| \nabla u\|_{m(\cdot)} \quad \text{for all} \ u \in W^{1,m(\cdot)}_0(\Omega).
\end{align*}
In the sequel, given $m \in C(\close)$ such that $1 < m(x) < N$ for all $x \in \close$, we denote by $m^*(x)$ the critical variable Sobolev exponent corresponding to $m$, that is,
\begin{align*}
	m^*(x) : = \frac{N \, m(x)}{N - m(x)} \quad \text{for all} \ x \in \close.
\end{align*}
Also, on $\partial\Omega$ we consider the $(N-1)$-dimensional Hausdorff (surface) measure $\sigma(\cdot)$ and, using this measure, we define in the usual way the boundary Lebesgue spaces $L^{m(\cdot)}(\partial\Omega)$.

Let $C^{0, \frac{1}{|\log t|}}(\close)$ be the set of all functions $w: \close \rightarrow \R$ that are log-H\"{o}lder
continuous, that is, there exists $d_0 > 0$ such that
\begin{align} \label{log}
	|w(x) - w(y)| \leq \frac{d_0}{|\log |x - y||} \quad \text{for all} \ x, y \in \close \ \text{with} \ |x - y| < \frac{1}{2}.
\end{align}

Then, the following embedding results for variable exponent Sobolev spaces are valid, see for example Diening-Harjulehto-Hasto-Ruzicka \cite[Corollary $8.3.2$]{Diening-Harjulehto-Hasto-Ruzicka-2011} and Fan \cite[Proposition $2.1$]{Fan-2010}.

\begin{proposition} \label{0log}
	Let $m \in C(\close) \,\cap\, C^{0, \frac{1}{|\log t|}}(\close)$ be such that $m(x)>1$ for all $x\in\close$. In addition, let $l \in C(\close)$ be such that $1 \leq l(x) \leq m^*(x)$ for all $x \in \close$. Then, we have the continuous embedding
	\begin{align*}
		W^{1,m(\cdot)}(\Omega) \hookrightarrow L^{l(\cdot)} (\Omega).
	\end{align*}
	Further, if $1 \leq l(x) < m^*(x)$ for all $x \in \close$ then the above embedding is compact.
\end{proposition}

\begin{proposition} \label{gamma}
	Let $m \in C(\close) \cap W^{1, \gamma} (\Omega)$ for some $\gamma > N$ with  $m(x)>1$ for all $x\in\close$. In addition, let $l \in C(\close)$ be such that $1 \leq l(x) \leq m^*(x)$ for all $x \in \close$. Then, we have the continuous embedding
	\begin{align*}
		W^{1,m(\cdot)} (\Omega) \hookrightarrow L^{l(\cdot)} (\partial\Omega).
	\end{align*}
	Further, if $1 \leq l(x) < m^*(x)$ for all $x \in \close$ then the above embedding is compact.
\end{proposition}

Next, we give some definitions which is need to introduce the Musielak-Orlicz spaces.

\begin{definition} \label{largedef} We have the following:
 \begin{enumerate}
	\item[\textnormal{(i)}] A continuous and convex function $\xi: [0, + \infty) \rightarrow [0, + \infty)$ is said to be a $\xi$-function if $\xi(0) = 0$ and $\xi(t) > 0$ for all $t > 0$.	
\item[\textnormal{(ii)}]
	A function $\xi: \Omega \times [0, + \infty) \rightarrow [0, + \infty)$ is said to be a generalized $\xi$-function if $\xi(\cdot, t)$ is measurable for all $t \geq 0$ and $\xi(x, \cdot)$ is a $\xi$-function for a.a. $x \in \Omega$. We denote the set of all generalized $\xi$-functions on $\Omega$ with $\Xi(\Omega)$.
	\item[\textnormal{(iii)}]
	A function $\xi \in \Xi(\Omega)$ is locally integrable if $\xi(\cdot, t) \in L^1(\Omega)$ for all $t > 0$.
	\item[\textnormal{(iv)}]
A function $\xi \in \Xi(\Omega)$ satisfies the $\Delta_2$-condition
	 if there exist a positive constant $a_1$ and a
	nonnegative function $h_1 \in L^1(\Omega)$ such that
	\begin{align*}
		\xi(x, 2t) \leq a_1 \xi(x, t) + h_1(x)
	\end{align*}
	for a.a. $x \in \Omega$ and all $t \geq 0$.
	\item[\textnormal{(v)}]
	Given $\xi, \zeta \in \Xi(\Omega)$ we say that $\xi$ is weaker than $\zeta$ (and we write $\xi \prec \zeta$), if there exist two
	positive constants $a_2, a_3$ and a nonnegative function $h_2 \in L^1(\Omega)$ such that
	\begin{align*}
		\xi(x, t) \leq a_2 \zeta\left(x, a_3 t\right) + h_2 (x)
	\end{align*}
for a.a. $x \in \Omega$ and all $t \geq 0$.
\item[\textnormal{(vi)}]
	A function $\xi: [0, + \infty) \rightarrow [0, + \infty)$ is said $N$-function if it is a $\xi$-function such that
\begin{align*}
	\lim_{t \to 0^+}\frac{\xi(t)}{t} = 0 \quad \text{and} \quad 		\lim_{t \to  + \infty}\frac{\xi(t)}{t} = + \infty.
\end{align*}
\item[\textnormal{(vii)}]
A function $\xi: \Omega \times [0, + \infty) \rightarrow [0, + \infty)$ is said to be a generalized $N$-function if $\xi(\cdot, t)$ is measurable for all $t \geq 0$ and $\xi(x, \cdot)$ is a $N$-function for a.a. $x \in \Omega$. We denote the set of all generalized $N$-functions on $\Omega$ with $N(\Omega)$.
\item[\textnormal{(viii)}]
	Let $\xi, \zeta \in N(\Omega)$. The function $\xi$ increases essentially slower than $\zeta$ near infinity, if for all $a_4 > 0$
\begin{align*}
	\lim_{t \to  + \infty} \frac{\xi\left(x, a_4 t\right)}{\zeta(x, t)} = 0 \quad \text{uniformly for a.a.} \ x \in \Omega.
\end{align*}
\item[\textnormal{(ix)}]
	A function $\xi \in N(\Omega)$ is said to be uniformly convex if for every $\eps > 0$ there exists $\eta_{\eps} > 0$ such that
\begin{align*}
	|t - s| \leq \eps \max \{t, s\} \quad \text{or} \quad \xi \l(x, \frac{t + s}{2}\r) \leq \l(1 - \eta_{\eps}\r) \frac{\xi(x, t) + \xi(x, s)}{2}
\end{align*}
for all $t, s \geq 0$ and for a.a. $x \in \Omega$.
	\end{enumerate}
\end{definition}
Finally, for a given $\xi \in \Xi(\Omega)$ we define the Musielak-Orlicz space $\Lp{\xi}$ by
\begin{align*}
	\Lp{\xi}
	=\left \{u\in M(\Omega)\,:\, \text{there exists} \, \alpha > 0 \ \text{so that} \ \rho_{\xi}(\alpha \,u)<+\infty \right \}
\end{align*}
equipped with the norm
\begin{align*}
	\|u\|_{\xi} := \inf \left \{ \tilde{\alpha} >0 \,:\, \rho_{\xi}\left(\frac{u}{\tilde{\alpha}}\right) \leq 1  \right \},
\end{align*}
where
\begin{align} \label{modular}
	\rho_{\xi}(u):=\into \xi(x,|u|)\,\diff x
\end{align}
is the corresponding modular. Then, the Musielak-Orlicz Sobolev space $W^{1,\xi}(\Omega)$ is given by
\begin{align*}
	W^{1,\xi}(\Omega)= \left \{u \in L^{\xi}(\Omega) \,:\, |\nabla u| \in L^{\xi}(\Omega) \right\}
\end{align*}
endowed with the norm
\begin{align*}
	\|u\|_{1,\xi}:= \|\nabla u \|_{\xi}+\|u\|_{\xi},
\end{align*}
where $\|\nabla u\|_{\xi}:=\|\,|\nabla u|\,\|_{\xi}$. Further, for a locally integrable $\xi \in N(\Omega)$ we denote by $W^{1, \xi}_0(\Omega)$ the completion of $C^\infty_0(\Omega)$ in $W^{1, \xi}(\Omega)$.

We recall that for the Musielak-Orlicz space $\Lp{\xi}$ we have the following results, see for example Musielak \cite[Theorem $7.7$, Theorem $8.5$ and Theorem $8.13$]{Musielak-1983}.
\begin{proposition} \label{2.6} Let $\xi \in \Xi(\Omega)$. We have the following:
		\begin{enumerate}
		\item[\textnormal{(i)}]
	The Musielak-Orlicz space $\Lp{\xi}$ is complete with respect to the norm $\|\cdot\|_{\xi}$.
		\item[\textnormal{(ii)}]
		Let $\xi, \zeta \in \Xi(\Omega)$ be locally integrable with $\xi \prec \zeta$ (see Definition \ref{largedef}\,$(v)$), then we have the continuous embedding
		\begin{align*}
			\Lp{\zeta} \hookrightarrow \Lp{\xi}.
		\end{align*}
		\item[\textnormal{(iii)}] If $\xi$ satisfies the $\Delta_2$-condition (see Definition \ref{largedef}\,$(iv)$), then
	\begin{align*}
		\Lp{\xi}
		=\left \{u\in M(\Omega)\,:\, \rho_{\xi}(u)<+\infty \right \}.
	\end{align*}
	\item[\textnormal{(iv)}] If $u \in \Lp{\xi}$, then $\rho_{\xi}(u) < 1$ (resp. $= 1, > 1$) if and only if $\|u\|_{\xi}< 1$ (resp. $= 1, > 1$).
		\end{enumerate}
\end{proposition}

Also, the following result holds for the Musielak-Orlicz Sobolev spaces $ W^{1,\xi}(\Omega)$ and $W_0^{1,\xi}(\Omega)$, see Musielak \cite[Theorem $10.2$]{Musielak-1983}
and Fan \cite[Propositions $1.7$
and $1.8$]{Fan-2012}.

\begin{theorem} \label{2.10}
	Let $\xi \in N(\Omega)$ be locally integrable such that
	\begin{align}\label{inf}
		\inf_{x \in \Omega} \xi(x, 1) > 0.
	\end{align}
Then both the spaces $ W^{1,\xi}(\Omega)$ and $W_0^{1,\xi}(\Omega)$ are separable Banach spaces which are reflexive if $\Lp{\xi}$ is reflexive.
\end{theorem}

\section{A new Musielak-Orlicz Sobolev space}\label{sec3}

\quad\,
Let $\Omega \subseteq \mathbb{R}^N$ ($N \geq 2$) be a bounded domain with Lipschitz boundary $\partial \Omega$. Let $p(\cdot), \,q(\cdot), \,r(\cdot)$ and $\mu_1(\cdot), \, \mu_2(\cdot)$ be functions satisfying the following assumptions:
\\

\begin{enumerate}[label=\textnormal{(H$1$)},ref=\textnormal{H$1$}]
	\item\label{H1}
	$p, q, r \in C(\close)$ are such that $$1 < p(x) <N, \quad p(x) <q(x) < r(x) < p^*(x) $$ for all $x\in\close$, where $p^*(x):= \frac{N p(x)}{N-p(x)}$ is the critical exponent corresponding to $p(x)$, and $\mu_1(\cdot), \, \mu_2(\cdot) \in\Linf$ are such that $$ \mu_1(x)\geq 0 \quad \text{and} \quad \mu_2(x)  \geq 0$$
	for all $x\in\close$.
\end{enumerate}
Under hypotheses \eqref{H1}, we introduce the nonlinear function
\begin{align*}
	\mathcal{T} \colon \Omega \times [0, +\infty) \to [0, +\infty)
\end{align*}
defined by
\begin{align*}
	\mathcal{T}(x, t) =t^{p(x)} +\mu_1(x) t^{q(x)} +\mu_2(x) t^{r(x)}
\end{align*}
for all $ x\in \Omega $ and for all $ t \geq 0$. We point out that
 $\mathcal{T}$ is a locally integrable, generalized $N$-function (see Definition \ref{largedef} (iii), (vii)) satisfying \eqref{inf}. Further, $\mathcal{T}$ satisfies the $\Delta_2$-condition with constant $a_1:= C_{\Delta} \leq 2^{r^+}$ (see Definition \ref{largedef} (iv)). In fact, we have
\begin{align*}
	\mathcal{T}(x, 2t) = (2t)^{p(x)} +\mu_1(x) (2t)^{q(x)} +\mu_2(x) (2t)^{r(x)} < 2^{r^+} \mathcal{T}(x, t)
\end{align*}
for all $ x\in \Omega $ and for all $ t \geq 0$. Finally, we remark that
\begin{align} \label{additivity}
	\mathcal{T}(x, t + s) \leq C_{\Delta} \, [\mathcal{T}(x, t) + \mathcal{T}(x, s)]
\end{align}
is valid for all $ x\in \Omega $ and for all $ t, s \geq 0$.

For way of this, according to Proposition \ref{2.6} (iii), the Musielak-Orlicz space $L^{\mathcal{T}}(\Omega)$ is given by
\begin{align*}
	\Lp{\mathcal{T}}
	=\left \{u\in M(\Omega)\,:\, \rho_{\mathcal{T}}(u)<+\infty \right \}
\end{align*}
equipped with the Luxemburg norm
\begin{align*}
	\|u\|_{\mathcal{T}} := \inf \left \{ \alpha >0 \,:\, \rho_{\mathcal{T}}\left(\frac{u}{\alpha}\right) \leq 1  \right \},
\end{align*}
where the modular $\rho_{\mathcal{T}}(\cdot)$ is defined by
\begin{align*}
	\rho_{\mathcal{T}}(u) :=\into \mathcal{T}(x,|u|)\,\diff x=\into \l(|u|^{p(x)}+ \mu_1(x)|u|^{q(x)} + \mu_2(x)|u|^{r(x)}\r)\,\diff x
\end{align*}
(see \eqref{modular}). Using the Musielak-Orlicz space $L^{\mathcal{T}}(\Omega)$, we define the corresponding Musielak-Orlicz Sobo\-lev space $W^{1,\mathcal{T}}(\Omega)$ by
\begin{align*}
	W^{1,\mathcal{T}}(\Omega)= \left \{u \in L^{\mathcal{T}}(\Omega) \,:\, |\nabla u| \in L^{\mathcal{T}}(\Omega) \right\}
\end{align*}
endowed with the norm
\begin{align*}
	\|u\|_{1,\mathcal{T}}:= \|\nabla u \|_{\mathcal{T}}+\|u\|_{\mathcal{T}},
\end{align*}
where $\|\nabla u\|_{\mathcal{T}}:=\|\,|\nabla u|\,\|_{\mathcal{T}}$. In addition, as $\mathcal{T}$ is a locally integrable generalized $N$-function, we denote by $W^{1,\mathcal{T}}_0(\Omega)$ the completion of $C^\infty_0(\Omega)$ in $W^{1,\mathcal{T}}(\Omega)$.

Our aim here is to study the properties of spaces $\Lp{\mathcal{T}}$, $W^{1,\mathcal{T}}(\Omega)$ and $W^{1,\mathcal{T}}_0(\Omega)$.
We in particular extend the results of Crespo-Blanco-Gasinski-Harjulehto-Winkert \cite{Crespo-Blanco-Gasinski-Harjulehto-Winkert-2022}, concerning the properties of the function space,
to the variable exponent multi phase case.

At first, we show that the spaces $\Lp{\mathcal{T}}$,  $W^{1,\mathcal{T}}(\Omega)$ and $W^{1,\mathcal{T}}_0(\Omega)$ are reflexive Banach
spaces. We refer the reader to Colasuonno-Squassina \cite[Proposition $2.14$]{Colasuonno-Squassina-2016} for a similar result in the case of two constant exponents.

\begin{proposition} \label{prop1}
	Let hypotheses \eqref{H1} be satisfied. Then, the spaces $\Lp{\mathcal{T}}$, $W^{1,\mathcal{T}}(\Omega)$ and $W^{1,\mathcal{T}}_0(\Omega)$ are reflexive Banach
	spaces. Moreover, for any sequence $\{u_n\}_{n\in\N} \subseteq \Lp{\mathcal{T}}$ such that
		\begin{align*}
		u_n \weak u \quad \text{in } \Lp{\mathcal{T}}
		\quad\text{and}\quad
		\rho_{\mathcal{T}}\left(u_n\right) \to \rho_{\mathcal{T}}(u)
	\end{align*}
we have that $u_n \to u$ in $ \Lp{\mathcal{T}} $.
\end{proposition}
\begin{proof} We recall that Proposition \ref{2.6} $(i)$ and Theorem \ref{2.10} ensure that $\Lp{\mathcal{T}}$, $W^{1,\mathcal{T}}(\Omega)$ and $W^{1,\mathcal{T}}_0(\Omega)$ are Banach spaces. Now, in order to prove that such spaces are also reflexive it is sufficient to show that $\Lp{\mathcal{T}}$ is uniformly convex. This for way of the Milman-Pettis theorem which asserts that a uniformly convex Banach space is
always reflexive, see Papageorgiou-Winkert \cite[Theorem $3.4.28$]{Papageorgiou-Winkert-2018}. In addition, from the uniform convexity of $\Lp{\mathcal{T}}$, thanks to Theorem \ref{2.10}, we have that $W^{1,\mathcal{T}}(\Omega)$ and $W^{1,\mathcal{T}}_0(\Omega)$ are reflexive.
	
We point out that in order to establish the uniform convexity of $\Lp{\mathcal{T}}$, it is sufficient to prove that the
$N$-function $\mathcal{T}$ is uniformly convex, see Diening-Harjulehto-H\"{a}st\"{o}-R$\mathring{\text{u}}$\v{z}i\v{c}ka \cite[Definition 2.4.5, Theorems 2.4.11 and 2.4.14]{Diening-Harjulehto-Hasto-Ruzicka-2011}. Next, thanks to the results in \cite[Lemma $2.4.17$ and Remark $2.4.19$]{Diening-Harjulehto-Hasto-Ruzicka-2011}, we also get the second assertion in our proposition.
	
Therefore, we show that $\mathcal{T}$ is uniformly convex. With a view to
Definition \ref{largedef} (ix), let $\varepsilon > 0$ and $t, s \geq 0$ be such that $|t - s| > \varepsilon \max\{t, s \}$. Taking into account that the function $t \rightarrow t^m$ is uniformly convex for all $m > 1$ (see \cite[Remark
	2.4.6]{Diening-Harjulehto-Hasto-Ruzicka-2011}) and as $p(x) > 1$ for all $x \in \close$ (see \eqref{H1}), we know that there exists $\eta_{\eps, p^{-}} > 0$ such that
	\begin{align*}
	\l(\frac{t+s}{2}\r)^{p^-} \leq \left(1 - \eta_{\eps, p^{-}}\right) \frac{t^{p^-} + s^{p^-}}{2}.
\end{align*}
Further, thanks to the convexity of the function $t \rightarrow t^{\frac{p(x)}{p^-}} $, we have
	\begin{align*}
	\l(\frac{t+s}{2}\r)^{p(x)} \leq \l[\left(1 - \eta_{\eps, p^{-}}\right) \frac{t^{p^-} + s^{p^-}}{2}\r]^{\frac{p(x)}{p^-}} \leq (1 - \eta_{\eps, p^{-}}) \frac{t^{p(x)} + s^{p(x)}}{2}
\end{align*}
for all $x \in \Omega$.

Reasoning in a similar way for $q(\cdot)$ and $r(\cdot)$, we also get
	\begin{align*}
	\l(\frac{t+s}{2}\r)^{q(x)} \leq \left(1 - \eta_{\eps, q^{-}}\right) \frac{t^{q(x)} + s^{q(x)}}{2}
\end{align*}
and
\begin{align*}
	\l(\frac{t+s}{2}\r)^{r(x)} \leq \left(1 - \eta_{\eps, r^{-}}\right) \frac{t^{r(x)} + s^{r(x)}}{2}
\end{align*}
for some $\eta_{\eps, q^{-}} > 0$ and $\eta_{\eps, r^{-}} > 0$, respectively.

Consequently, if we put $\eta_{\eps} = \min\left\{\eta_{\eps, p^{-}}, \eta_{\eps, q^{-}}, \eta_{\eps, r^{-}}\right\}$, then we have
\begin{align*}
		 \mathcal{T} \l(x, \frac{t+s}{2}\r) &=
		\l(\frac{t+s}{2}\r)^{p(x)} + \mu_1 \l(\frac{t+s}{2}\r)^{q(x)} + \mu_2 \l(\frac{t+s}{2}\r)^{r(x)}\\
		& \leq \left(1 - \eta_{\eps}\right) \ \frac{t^{p(x)} + \mu_1 t^{q(x)} + \mu_2 t^{r(x)} + s^{p(x)} + \mu_1 s^{q(x)} + \mu_2 s^{r(x)}}{2}\\
		& = (1 - \eta_{\eps}) \ \frac{\mathcal{T}(x, t) + \mathcal{T}(x, s)}{2}.
\end{align*}
This guarantees that $\mathcal{T}$ is uniformly convex (see Definition \ref{largedef} (ix)) and so we have the claim.
\end{proof}

Next, we show that analogous results to the ones in Proposition \ref{Prop0} are valid for the norm $\|\cdot\|_{\mathcal{T}}$ and the corresponding modular $\rho_{\mathcal{T}}(\cdot)$. We mention Liu-Dai \cite[Proposition $2.1$]{Liu-Dai-2018} for a similar result in the case of two constant exponents.

\begin{proposition}\label{Prop2}
	Let hypotheses \eqref{H1} be satisfied. Then the following hold:
	\begin{enumerate}
		\item[\textnormal{(i)}]
		if $u \neq 0$, we have $\|u\|_{\mathcal{T}} = \lambda $ if and only if $\rho_{\mathcal{T}}\left(u/\lambda\right) = 1$;
		\item[\textnormal{(ii)}]
		$\|u\|_{\mathcal{T}} < 1$ (resp. $> 1, = 1$) if and only if $\rho_{\mathcal{T}}(u) < 1$ (resp. $> 1, = 1$);
		\item[\textnormal{(iii)}]
		if $\|u\|_{\mathcal{T}} < 1$ then $\|u\|_{\mathcal{T}}^{r^+} \leq \rho_{\mathcal{T}} (u) \leq \|u\|_{\mathcal{T}}^{p^-}$;
		\item[\textnormal{(iv)}]
		if $\|u\|_{\mathcal{T}} > 1$ then $\|u\|_{\mathcal{T}}^{p^-} \leq \rho_{\mathcal{T}} (u) \leq \|u\|_{\mathcal{T}}^{r^+}$;
		\item[\textnormal{(v)}]
		$\|u\|_{\mathcal{T}} \rightarrow 0$ if and only if $\rho_{\mathcal{T}} (u) \rightarrow 0$;
		\item[\textnormal{(vi)}]
		$\|u\|_{\mathcal{T}} \rightarrow + \infty$ if and only if $\rho_{\mathcal{T}} (u) \rightarrow + \infty$;
		\item[\textnormal{(vii)}]
		$\|u\|_{\mathcal{T}} \rightarrow 1$ if and only if $\rho_{\mathcal{T}} (u) \rightarrow 1$;
		\item[\textnormal{(viii)}]
		if $u_n \rightarrow u$ in $\Lp{\mathcal{T}}$ we have $\rho_{\mathcal{T}} \left(u_n\right) \rightarrow \rho_{\mathcal{T}} (u)$.
	\end{enumerate}
\end{proposition}
\begin{proof} For a given $u \in \Lp{\mathcal{T}}$, we have that the function $\rho_{\mathcal{T}} (\beta u)$ is continuous, convex and even
	in the variable $\beta$. Further, $\rho_{\mathcal{T}} (\beta u)$ is strictly increasing when $\beta \in [0, + \infty)$. Hence, from the definition of $\|\cdot\|_{\mathcal{T}}$, we easily deduce that $\|u\|_{\mathcal{T}} = \lambda $ if and only if $\rho_{\mathcal{T}}(\frac{u}{\lambda}) = 1$, that is, $(i)$ holds.
	
Now, taking into account that $(ii)$ follows from $(i)$, we prove $(iii)$. We remark that for a given $\beta > 1$ the following inequality
	\begin{align*}
		\beta^{p^-} \rho_{\mathcal{T}}(u) \leq \rho_{\mathcal{T}}(\beta u) \leq \beta^{r^+} \rho_{\mathcal{T}}(u)
	\end{align*}
is valid for all $u \in \Lp{\mathcal{T}}$. With a view to this, we take $u \in \Lp{\mathcal{T}}$ such that $\|u\|_{\mathcal{T}} = \lambda $ with $0 < \lambda < 1$. Then, thanks to $(i)$ we have $\rho_{\mathcal{T}}\left( u/\lambda\right) = 1$. Further, since $ 1/\lambda > 1$ we have
\begin{align*}
		 \frac{\rho_{\mathcal{T}}(u)}{\lambda^{p^-}} \leq \rho_{\mathcal{T}}\l(\frac{u}{\lambda}\r) = 1 \leq \frac{\rho_{\mathcal{T}}(u)}{\lambda^{r^+}}
\end{align*}
which gives us
\begin{align*}
\|u\|_{\mathcal{T}}^{r^+} =  \lambda^{r^+} \leq \rho_{\mathcal{T}}(u) \leq \lambda^{p^-} = \|u\|_{\mathcal{T}}^{p^-},
\end{align*}
that is, we deduce that $(iii)$ holds. Now, we point out that from $(iii)$ we can easily deduce $(v)$. In addition, we know that for a given $0 < \beta < 1$ the following inequality
\begin{align*}
	\beta^{r^+} \rho_{\mathcal{T}}(u) \leq \rho_{\mathcal{T}}(\beta u) \leq \beta^{p^-} \rho_{\mathcal{T}}(u)
\end{align*}
is valid for all $u \in \Lp{\mathcal{T}}$. Then, using this inequality and reasoning as done in order to show $(iii)$, we can obtain $(iv)$. Next, from $(iv)$ we easily deduce that $(vi)$ holds too. In addition, $(iii)$ and $(iv)$ permit to infer $(vii)$.

Therefore, we must only prove $(viii)$. To this end, we consider a sequence $\left\{u_n\right\}_{n \in \N} \subseteq \Lp{\mathcal{T}}$ such that $u_n \rightarrow u$ in $\Lp{\mathcal{T}}$. Thanks to $(v)$ we know that $\rho_{\mathcal{T}}\left(u_n - u\right) \to 0$. Hence, since all the addends in $\rho_{\mathcal{T}}(\cdot)$ are positive, we deduce that $\rho_{p(\cdot)}\left(u_n - u\right) \to 0$. Taking into account this, Proposition \ref{Prop0} $(v)$ and the continuous embedding $L^{p(\cdot)} (\Omega) \hookrightarrow L^{p^-} (\Omega) $ (see Section \ref{sec2}) give that $\left\Vert u_n - u\right\Vert_{p^-} \to 0$. Thus, we have that $u_n \to u$ a.\,e. through a subsequence. Now, we recall that
\begin{eqnarray}
\left\vert u_n\right\vert^{p(x)}+ \mu_1(x)\left\vert u_n\right\vert^{q(x)} + \mu_2(x)\left\vert u_n\right\vert^{r(x)}
& \leq& 2^{r^+} \l(\left\vert u_n - u\right\vert^{p(x)} + | u|^{p(x)}\r)\nonumber\\
&  &+ 2^{r^+} \l(\mu_1(x)\left\vert u_n - u\right\vert^{q(x)} + \mu_1(x)|u|^{q(x)}\r)\nonumber\\
&  &+ 2^{r^+} \l(\mu_2(x)\left\vert u_n - u\right\vert^{r(x)} + \mu_2(x)| u|^{r(x)} \r).\nonumber
\end{eqnarray}
Hence, since $\rho_{\mathcal{T}}\left(u_n - u\right) \to 0$ we can deduce that
\begin{align*}
	\left\{\left\vert u_n\right\vert^{p(x)}+ \mu_1(x)\left\vert u_n\right\vert^{q(x)} + \mu_2(x)\left\vert u_n\right\vert^{r(x)}\right\}_{n \in \N}
\end{align*}
is an uniformly integrable sequence which converges a.e. to
$$|u|^{p(x)}+ \mu_1(x)|u|^{q(x)} + \mu_2(x)|u|^{r(x)}$$ by the a.\,e. convergence of $u_n \to u$.
Next, using Vitali's Theorem (see Bogachev \cite[Theorem 4.5.4]{Bogachev-2007}), we deduce that $\rho_{\mathcal{T}}\left(u_n\right) \to \rho_{\mathcal{T}}(u)$ through this subsequence. So, taking into account that it is possible to recover the whole sequence by the
subsequence principle, we conclude that $(viii)$ holds.
\end{proof}
In what follows, we consider the spaces
\begin{align*}
	L^{q(\cdot)}_{\mu_1}(\Omega) = \left\{u\in M(\Omega)\, :\, \int_\Omega \mu_1(x) \, |u|^{q(x)}\,\diff x < +\infty\right\},
\end{align*}
and
\begin{align*}
	L^{r(\cdot)}_{\mu_2}(\Omega) = \left\{u\in M(\Omega)\, :\, \int_\Omega \mu_2(x) \, |u|^{r(x)}\,\diff x < +\infty\right\}
\end{align*}
furnished with the seminorms
\begin{align*}
	\|u\|_{{q(\cdot)}, \mu_1}= \inf  \left\{ \alpha_1 > 0 \,:\,   \int_\Omega \mu_1(x) \, \l(\dfrac{|u|}{\alpha_1}\r)^{q(x)} \, \diff x \leq 1 \right\}
\end{align*}
and
\begin{align*}
	\|u\|_{{r(\cdot)}, \mu_2}= \inf  \left\{ \alpha_2 > 0 \,:\,   \int_\Omega \mu_2(x) \, \l(\dfrac{|u|}{\alpha_2}\r)^{r(x)} \, \diff x \leq 1 \right\},
\end{align*}
respectively.

Now, we show that the classical Sobolev embedding results extend to our Musielak-Orlicz and Musielak-Orlicz Sobolev spaces as follows.
\begin{proposition} \label{embedding}
	Let hypotheses \eqref{H1} be satisfied. Then the following hold:
		\begin{enumerate}
		\item[\textnormal{(i)}]
		 $\Lp{\mathcal{T}} \hookrightarrow \Lp{m(\cdot)}$,  $\Wp{\mathcal{T}}\hookrightarrow \Wp{m(\cdot)}$, $\Wpzero{\mathcal{T}}\hookrightarrow \Wpzero{m(\cdot)}$ are continuous for all $m \in C(\close)$ with $1 \leq m(x) \leq p(x)$ for all $x \in \close$;
		 \item[\textnormal{(ii)}]
		 if $p \in C^{0, \frac{1}{|\log t|}}(\close)$, then
		 $\Wp{\mathcal{T}}\hookrightarrow \Lp{m(\cdot)}$ and $\Wpzero{\mathcal{T}} \hookrightarrow \Lp{m(\cdot)}$ are continuous for all $m \in C(\close) $ with $ 1 \leq m(x) \leq p^*(x)$ for all $x \in \close$;
		\item[\textnormal{(iii)}]
		$\Wp{\mathcal{T}} \hookrightarrow \Lp{m(\cdot)}$ and $\Wpzero{\mathcal{T}} \hookrightarrow \Lp{m(\cdot)}$ are compact for all $m \in C(\close) $ with $ 1 \leq m(x) < p^*(x)$ for all $x \in \close$;
		\item[\textnormal{(iv)}]
		if $p \in W^{1, \gamma}(\Omega)$ for some $\gamma > N$, then
		$\Wp{\mathcal{T}}\hookrightarrow L^{m(\cdot)}(\partial\Omega)$ and $\Wpzero{\mathcal{T}} \hookrightarrow L^{m(\cdot)}(\partial\Omega)$ are continuous for all $m \in C(\close) $ with $ 1 \leq m(x) \leq p^*(x)$ for all $x \in \close$;
	\item[\textnormal{(v)}]
		$\Wp{\mathcal{T}}\hookrightarrow L^{m(\cdot)}(\partial\Omega)$ and $\Wpzero{\mathcal{T}} \hookrightarrow L^{m(\cdot)}(\partial\Omega)$ are compact for all $m \in C(\close) $ with $ 1 \leq m(x) < p^*(x)$ for all $x \in \close$;
		\item[\textnormal{(vi)}]
		$\Lp{\mathcal{T}} \hookrightarrow L^{q(\cdot)}_{\mu_1} (\Omega)$ and $\Lp{\mathcal{T}} \hookrightarrow L^{r(\cdot)}_{\mu_2} (\Omega)$ are continuous;
		\item[\textnormal{(vii)}]
	$ \Lp{r(\cdot)} \hookrightarrow \Lp{\mathcal{T}}$ is continuous.
	\end{enumerate}
\end{proposition}
\begin{proof}
Taking into account that $\Omega$ is bounded, $(i)$ easily follows by the classical embedding results for variable exponent Lebesgue and Sobolev spaces.
 Indeed, we recall that the function $t \rightarrow t^{p(x)}$ for all $t \geq 0$ and all $x \in \close$ is weaker than $\mathcal{T}(x, t)$ (see Definition \ref{largedef} $(v)$). On account of this, thanks to Proposition \ref{2.6} $(ii)$, we know that $\Lp{\mathcal{T}} \hookrightarrow \Lp{p(\cdot)}$ is continuous. Then, applying the classical embedding results for variable exponent Lebesgue spaces we deduce that $\Lp{\mathcal{T}} \hookrightarrow \Lp{m(\cdot)}$ is continuous for all $m \in C(\close)$ with $1 \leq m(x)\leq p(x)$ for all $x \in \close$. Clearly, from this we in addition infer that the embeddings $\Wp{\mathcal{T}}\hookrightarrow \Wp{m(\cdot)}$ and $\Wpzero{\mathcal{T}}\hookrightarrow \Wpzero{m(\cdot)}$ are continuous.

Next, reasoning in a similar way as above but using Propositions \ref{0log} and \ref{gamma} instead of Proposition \ref{2.6}, we easily get that assertions $(ii)-(v)$ are also valid.

Now, we prove $(vi)$. Let $u \in \Lp{\mathcal{T}}$. We stress that
\begin{eqnarray}
	\int_\Omega \mu_1(x) \, |u(x)|^{q(x)} \, \diff x
	& \leq&\int_\Omega \l(|u(x)|^{p(x)} +  \mu_1(x) \, |u(x)|^{q(x)} + \mu_2(x) \, |u(x)|^{r(x)}\r) \, \diff x \nonumber \\
	&=& \rho_{\mathcal{T}}(u)\nonumber
\end{eqnarray}
and analogously,
\begin{align*}
	\int_\Omega \mu_2(x) \, |u(x)|^{r(x)} \, \diff x &\leq \rho_{\mathcal{T}}(u).
\end{align*}
Now, taking into account Proposition \ref{Prop2} $(i)$ we get
\begin{align*}
	\rho_{\mathcal{T}}\l(\frac{u}{\|u\|_{\mathcal{T}}}\r) = 1 \quad \text{whenever} \ u \neq 0,
\end{align*}
then we can affirm
\begin{align*}
	\int_\Omega \mu_1(x) \, \l(\frac{u(x)}{\|u\|_{\mathcal{T}}}\r)^{q(x)} \, \diff x \leq 1 \
\end{align*}
and
\begin{align*}
	\int_\Omega \mu_2(x) \, \l(\frac{u(x)}{\|u\|_{\mathcal{T}}}\r)^{r(x)} \, \diff x \leq 1
\end{align*}
for all $u \neq 0$. Hence, we conclude that
\begin{align*}
	\|u\|_{q(\cdot), \mu_1} \leq \|u\|_{\mathcal{T}} \quad \text{and} \quad \|u\|_{r(\cdot), \mu_2} \leq \|u\|_{\mathcal{T}}
\end{align*}
for all $u \in \Lp{\mathcal{T}}$. This clearly guarantees that $(vi)$ holds.

Lastly, we check $(vii)$. We note that
\begin{align*}
	\mathcal{T}(x, t) \leq & \ \l(1 + t^{r(x)}\r) + \mu_1(x) \l(1 + t^{r(x)}\r) + \mu_2(x) t^{r(x)}\\
	\leq & \ \l(1 + \left\Vert\mu_1\right\Vert_{\infty}\r) + \l(1 + \left\Vert\mu_1\right\Vert_{\infty} + \left\Vert\mu_2\right\Vert_{\infty}\r) t^{r(x)}
\end{align*}
for all $t \geq 0$ and for a.\,a. $x \in \Omega$. This clearly gives that
$\mathcal{T}(x, t) $ is weaker than the function $t \rightarrow t^{r(x)}$ for all $t \geq 0$ and all $x \in \close$ (see Definition \ref{largedef} $(v)$). Consequently, we can use Proposition \ref{2.6} $(ii)$ in order to conclude that the embedding $ \Lp{r(\cdot)} \hookrightarrow \Lp{\mathcal{T}}$ is continuous.

\end{proof}

Thanks to Proposition \ref{embedding}, we are able to state the following result.

\begin{proposition} \label{Poincare}
	Let hypotheses \eqref{H1} be satisfied. Then, the following hold:
	\begin{enumerate}
		\item[\textnormal{(i)}]
		$\Wp{\mathcal{T}} \hookrightarrow \Lp{\mathcal{T}}$ is a compact embedding;
		\item[\textnormal{(ii)}]
		there exists $\tilde{c} > 0$, independent of $u$, such that
		\begin{align*}
			\|u \|_{\mathcal{T}} \leq \tilde{c} \ \| \nabla u \|_{\mathcal{T}} \quad \text{for all} \ u \in \Wpzero{\mathcal{T}}.
		\end{align*}
			\end{enumerate}
\end{proposition}

\begin{proof}
	We point out that Proposition \ref{embedding} $(iii)$ assures that $\Wp{\mathcal{T}} \hookrightarrow \Lp{r(\cdot)}$ is compact. In addition, from Proposition \ref{embedding} $(vii)$ we know that $\Lp{r(\cdot)} \hookrightarrow \Lp{\mathcal{T}}$ is continuous. According of all this, we can clearly conclude that $(i)$ holds.
	
	Therefore, we have only to prove $(ii)$. Now, we reason for way of contradiction and thus we suppose that $(ii)$ is not true. Hence, we can affirm that there exists a sequence $\left\{u_n\right\}_{n\in\N} \subseteq \Wpzero{\mathcal{T}}$ such that
	\begin{align} \label{n}
		\left\Vert u_n \right\Vert_{\mathcal{T}} > n \left\Vert u_n \right\Vert_{\mathcal{T}} \quad \text{for all} \ n \in \N.
	\end{align}
We put $y_n = u_n/\left\Vert u_n \right\Vert_{\mathcal{T}}$ for all $n\in\N$. Clearly, we have that $\left\Vert u_n \right\Vert_{\mathcal{T}}= 1$ for all $n \in \mathbb{N}$. Furthermore, by \eqref{n} we know that
\begin{align*}
	1 \geq \frac{1}{n} > \left\Vert u_n \right\Vert_{\mathcal{T}} \quad \text{for all} \ n \in \N
\end{align*}
and so, we have that $\left\{y_n\right\}_{n\in\N}$ is bounded in $\Wpzero{\mathcal{T}}$. Therefore, we can assume that
\begin{align*}
	y_n \weak y \quad \text{in } \ \Wpzero{\mathcal{T}} \quad \text{for some} \ y \in \Wpzero{\mathcal{T}}.
\end{align*}
Hence, using the weak lower semicontinuity of the mapping $u \rightarrow \left\Vert \nabla u \right\Vert_{\mathcal{T}} $, we get
\begin{align*}
	\| \nabla y \|_{\mathcal{T}} \leq \liminf_{n \to \infty} \left\vert \nabla y_n  \right\vert_{\mathcal{T}} \leq \inf \frac{1}{n} = 0.
\end{align*}
This gives that $y$ is a constant function. Taking into account that $y \in W^{1, p(\cdot)}_0 (\Omega)$ (see Proposition \ref{embedding} $(i)$) and the only constant function in $W^{1, p(\cdot)}_0 (\Omega)$ is trivially the null function, then we conclude that $y = 0$.
Now, we note that by $(i)$ we know that
\begin{align*}
	y_n \to y \quad \text{in } \ \Lp{\mathcal{T}}
\end{align*}
which gives
\begin{align*}
	 \lim_{n \to \infty} \left\vert y_n  \right\vert_{\mathcal{T}} = \| y \|_{\mathcal{T}}.
\end{align*}
But $\lim_{n \to \infty}  \left\vert y_n  \right\vert_{\mathcal{T}} = 1$ and this clearly implies $y \neq 0$. Thus, we arrive to a contradiction and hence we conclude that $(ii)$ holds.
\end{proof}

According to $(ii)$ of Proposition \ref{Poincare}, we can furnish the space $\Wpzero{\mathcal{T}}$ with the equivalent norm
\begin{align*}
	\|u\|_{1, \mathcal{T}, 0}:= \|\nabla u \|_{\mathcal{T}}\quad\text{for all }  u \in \Wpzero{\mathcal{T}}.
\end{align*}
We emphasize that such a norm is uniformly convex. In addition, we have that the Radon-Riesz property with respect to the modular is satisfied. In fact, the following result holds true.

\begin{proposition} \label{Radon-Riesz}
		Let hypotheses \eqref{H1} be satisfied. Then, the following hold:
		\begin{enumerate}
			\item[\textnormal{(i)}]
		 the norm $\|u\|_{1, \mathcal{T}, 0}$ on $\Wpzero{\mathcal{T}}$ is uniformly convex;
			\item[\textnormal{(ii)}] for any sequence $\left\{u_n\right\}_{n\in\N} \subseteq \Wpzero{\mathcal{T}}$ such that
			\begin{align*}
				u_n \weak u \quad \text{in }  \Wpzero{\mathcal{T}} \quad \text{and} \quad \rho_{\mathcal{T}}\left(\nabla u_n\right) \to \rho_{\mathcal{T}}(\nabla u)
			\end{align*}
		we have that $u_n \to u$ in $\Wpzero{\mathcal{T}}$.
		\end{enumerate}
\end{proposition}
\begin{proof} We can establish the claim thanks to Theorems $2.4$ and $3.5$ of Fan-Guan \cite{Fan-Guan-2010}, arguing like in the proof of Proposition $2.19$ of Crespo-Blanco-Gasinski-Harjulehto-Winkert \cite{Crespo-Blanco-Gasinski-Harjulehto-Winkert-2022}. In particular, we stress that assertion $(i)$ can be obtained applying Theorem $2.4$, while assertion $(ii)$ can be derived by Theorem $3.5$.
\end{proof}

Here, for any $s \in \R$ we put $s_\pm = \max \{ \pm s, 0 \}$, that means $s = s_+ - s_-$ and $|s| = s_+ + s_-$. Also, for any function $v \colon \Omega \to \R$ we write $v_\pm (\cdot)$ by $ [v(\cdot)]_\pm$.

The following proposition states that the spaces $\Wp{\mathcal{T}}$ and $\Wpzero{\mathcal{T}}$ are closed with respect to truncations.
\begin{proposition}
	Let hypotheses \eqref{H1} be satisfied. Then, the following hold:
	\begin{enumerate}
		\item[\textnormal{(i)}]
		if $u \in \Wp{\mathcal{T}}$, then $u_\pm \in \Wp{\mathcal{T}}$ with $\nabla\left(u_\pm\right) = \nabla u \,\chi_{\{\pm u > 0\}}$;
		\item[\textnormal{(ii)}]
		if $u_n \rightarrow u$ in $\Wp{\mathcal{T}}$, then $ \left(u_n\right)_\pm \rightarrow u_\pm$ in $\Wp{\mathcal{T}}$;
		\item[\textnormal{(iii)}]
		if $u \in \Wpzero{\mathcal{T}}$, then $ u_\pm \in \Wpzero{\mathcal{T}}$.
	\end{enumerate}
\end{proposition}
\begin{proof}
	We point out that the claim follows using Proposition \ref{embedding} and arguing similarly to the proof of Proposition $2.17$ of Crespo-Blanco-Gasinski-Harjulehto-Winkert \cite{Crespo-Blanco-Gasinski-Harjulehto-Winkert-2022}. To be more precise,
	we obtain assertions $(i)$ and $(ii)$ using Proposition \ref{embedding} $(i)$. We establish that $(iii)$ holds using Proposition \ref{embedding} $(viii)$.
\end{proof}


\section{The variable exponent multi phase operator $A_\mathcal{T}$}\label{sec4}

In this section, we introduce and study the operator $A_\mathcal{T}$ which is a multi phase operator with variable exponents. We in particular investigate its regularity properties. Precisely, we prove that it is bounded, continuous, strictly monotone, coercive and further it satisfies the $\left(S_+\right)$-property.

Before of defining the new operator, we recall some notions that will be used in the sequel. Here, given a Banach space $X$ and its dual space $X^*$, we denote by $\langle \cdot \,, \cdot\rangle$ the duality pairing between $X$ and $X^*$.

\begin{definition}\label{coercive}
	
Let $X$ be a reflexive Banach space and let $A\colon X\to X^*$. Then $A$ is called
	\begin{enumerate}
		\item[\textnormal{(i)}]
		to satisfy the $\Ss$-property if
		\begin{align*}
			u_n \weak u \ \text{in} \ X \quad \text{and} \quad \limsup_{n\to + \infty}\, \left\langle A\left(u_n\right),u_n-u\right\rangle \leq 0
		\end{align*}
		imply $u_n\to u$ in $X$;
		\item[\textnormal{(ii)}]
		pseudomonotone if
		\begin{align*}
			u_n \weak u \ \text{in} \ X \quad \text{and} \quad \limsup_{n\to + \infty}\, \left\langle A\left(u_n\right),u_n-u\right\rangle \leq 0
		\end{align*} imply
		\begin{align*}
			\liminf_{n \to +\infty} \left\langle A\left(u_n\right),u_n-v\right\rangle \geq \langle A(u) , u-v\rangle
		\end{align*}
	for all $v \in X$;
		\item[\textnormal{(iii)}]
		coercive if
		\begin{align*}
			\lim_{\|u\|_X\to +\infty }\frac{\langle A(u), u \rangle}{\|u\|_X}=+\infty.
		\end{align*}
	\end{enumerate}
\end{definition}

\begin{remark} \label{RemarkPseudomonotone}
	If the operator $A\colon X \to X^*$ is bounded, then the definition of pseudomonotone operator given in Definition \ref{coercive} $(ii)$ is equivalent to the following:
	
	the operator $A$ is pseudomonotone if
	\begin{align*}
		u_n \weak u \ \text{in} \ X \quad \text{and} \quad \limsup_{n\to \infty}\, \left\langle A\left(u_n\right),u_n-u\right\rangle \leq 0
	\end{align*}
imply
\begin{align*}
A\left(u_n\right) \weak A(u) \quad \text{and} \quad \left\langle A\left(u_n\right),u_n\right\rangle \to \langle A(u),u\rangle.
\end{align*}
\end{remark}

Also, we remark the important surjectivity property for pseudomonotone operators, see, for example, Papageorgiou-Winkert \cite[Theorem 6.1.57]{Papageorgiou-Winkert-2018}.

\begin{theorem}\label{theorem_pseudomonotone}
	Let $X$ be a real reflexive Banach space. In addition, let $A\colon X\to X^*$ be a pseudomonotone bounded coercive operator and $z \in X^*$. Then, there exists a solution of the equation $Au=z$.
\end{theorem}

Now, we consider the nonlinear operator
$A_{\mathcal{T}}\colon \Wpzero{\mathcal{T}}\to \Wpzero{\mathcal{T}}^*$ defined by
\begin{align*}
	\begin{split}
		\left\langle A_{\mathcal{T}}(u),v\right\rangle &=\\
		&\into \left(|\nabla u|^{p(x)-2}\nabla u +\mu_1(x)|\nabla u|^{q(x)-2}\nabla u + \mu_2(x)|\nabla u|^{r(x)-2}\nabla u \right)\cdot\nabla v \,\diff x
	\end{split}
\end{align*}
for all $u,v\in \Wpzero{\mathcal{T}}$. Thus, the energy functional $I_{\mathcal{T}}\colon \Wpzero{\mathcal{T}}\to \mathbb{R} $ related to $A_{\mathcal{T}}$ is given by
\begin{align} \label{Itau}
	I_{\mathcal{T}}(u) := \int_\Omega \l(\frac{1}{p(x)} |\nabla u|^{p(x)} + \frac{\mu_1}{q(x)} |\nabla u|^{q(x)} + \frac{\mu_2}{r(x)} |\nabla u|^{r(x)} \r) \,\diff x
\end{align}
for all $u \in \Wpzero{\mathcal{T}}$.

Using Proposition \ref{embedding}, we are able to exhibit the following properties of $I_{\mathcal{T}}$.

\begin{proposition} \label{propertiesI}
		Let hypotheses \eqref{H1} be satisfied. Then, the functional $I_{\mathcal{T}}$ is well-defined and of class $C^1$ with $I_{\mathcal{T}}^{\prime}(u) = A_{\mathcal{T}}(u)$.
\end{proposition}
\begin{proof} At first, we note that
	\begin{align*}
		0 \leq \frac{\rho_{\mathcal{T}}(\nabla u)}{r^+} \leq I_{\mathcal{T}}(u) \leq  \frac{\rho_{\mathcal{T}}(\nabla u)}{p^-} < + \infty
	\end{align*}
and thus, the functional $I_{\mathcal{T}}$ is well-defined. Now, similar to the proof of
Proposition 3.1 in Crespo-Blanco-Gasinski-Harjulehto-Winkert \cite{Crespo-Blanco-Gasinski-Harjulehto-Winkert-2022}, using the Dominated Convergence Theorem, Proposition \ref{embedding} $(i)$, the H\"{o}lder inequality and Proposition \ref{Prop0} $(iii)$, $ (iv)$, we can show
that for $t \to 0$ we have
\begin{align*}
	\int_\Omega\frac{|\nabla u+ t \,\nabla h|^{p(x)} - |\nabla u|^{p(x)}}{t p(x)} \, dx \to \int_\Omega |\nabla u|^{p(x) - 2} \nabla u \cdot \nabla h \,\diff x.
\end{align*}
Further, the same arguments give us the following, for $t \to 0$
\begin{align*}
	\int_\Omega\frac{  \mu_1(x) | \nabla u + t \,\nabla h|^{q(x)} - \mu_1(x) | \nabla u|^{q(x)}}{t q(x)} \, dx \to \int_\Omega \mu_1(x) |\nabla u|^{q(x) - 2} \nabla u \cdot \nabla h \,\diff x
\end{align*}
and
\begin{align*}
	\int_\Omega\frac{  \mu_2(x) | \nabla u + t \,\nabla h|^{r(x)} - \mu_2(x) | \nabla u|^{r(x)}}{t r(x)} \, dx \to \int_\Omega \mu_2(x) |\nabla u|^{r(x) - 2} \nabla u \cdot \nabla h \,\diff x.
\end{align*}
We stress that in the case of the variable exponents $q(x)$ and $r(x)$, it is sufficient to
apply $(vi)$ of Proposition \ref{embedding} instead of $(i)$ and, in order to use H\"{o}lder inequality, to split
$\mu_1(x) = \mu_1(x)^{\frac{1}{q(x)}} \mu_1(x)^{\frac{q(x) -1}{q(x)}}$ and
$\mu_2(x) = \mu_2(x)^{\frac{1}{r(x)}} \mu_2(x)^{\frac{r(x) -1}{r(x)}}$.
According of all this, we conclude that the Gateaux derivative of $I_{\mathcal{T}}$ is given by $A_{\mathcal{T}}$. Then, in a similar way, using again Proposition \ref{embedding}, the H\"{o}lder inequality, Proposition \ref{Prop0} and in addition the Vitali Theorem, we also get that $I_{\mathcal{T}}$ is of class $C^1$.

\end{proof}

 Finally, the properties of the operator $A_{\mathcal{T}}\colon \Wpzero{\mathcal{T}}\to \Wpzero{\mathcal{T}}^*$ are summarized in the next proposition. We refer the reader to Liu-Dai \cite[Proposition 3.1]{Liu-Dai-2018} for
 the analogous result in the case of two constant exponents.

\begin{proposition}\label{PropertiesA}
	Let hypothesis \eqref{H1} be satisfied. Then, the operator $A_{\mathcal{T}}$
	\begin{enumerate}
		\item[\textnormal{(i)}] is bounded (that is, it maps bounded sets into bounded sets), continuous, strictly monotone;
	\item[\textnormal{(ii)}]  satisfies the
	$\left(S_+\right)$-property, that is,
	\begin{align*}
		u_n \weak u \quad \text{in }  \Wpzero{\mathcal{T}} \quad \text{and} \quad \limsup_{n \to + \infty} \, \left\langle A\left(u_n\right), u_n - u\right\rangle \leq 0
	\end{align*}
imply $u_n \to u \quad \text{in }  \Wpzero{\mathcal{T}}$;
	\item[\textnormal{(iii)}] is coercive and a homeomorphism.
	\end{enumerate}
\end{proposition}
\begin{proof}
	We start by assertion $(i)$. We remark that $A_{\mathcal{T}}$ is the Gateaux derivative of the functional $I_{\mathcal{T}}$ defined in \eqref{Itau}. Further, thanks to Proposition \ref{propertiesI} we know that $I_{\mathcal{T}}$ is of class $C^1$. Hence, we clearly deduce that $A_{\mathcal{T}}$ is continuous. Also, we have that
	\begin{align*}
	\left\langle A_{\mathcal{T}}(u) - A_{\mathcal{T}}(v), u - v\right\rangle = & \int_\Omega \left(|\nabla u|^{p(x) -2}\nabla u - |\nabla v|^{p(x) -2}\nabla v\right) \cdot (\nabla u - \nabla v) \, \diff x\\
	 & +  \int_\Omega \mu_1(x) \left(|\nabla u|^{q(x) -2}\nabla u - |\nabla v|^{q(x) -2}\nabla v\right) \cdot (\nabla u - \nabla v) \, \diff x\\
	 & +  \int_\Omega \mu_2(x) \left(|\nabla u|^{r(x) -2}\nabla u - |\nabla v|^{r(x) -2}\nabla v\right) \cdot (\nabla u - \nabla v) \, \diff x.
	\end{align*}
We stress that all the addends in the previous equality are major of zero whenever $u \neq v$. We in fact know that the inequality
\begin{align*}
	\left(\left\vert z_1\right\vert^{s -2} z_1 -\left\vert z_2\right\vert^{s -2} z_2\right) \cdot \left(z_1 - z_2\right) > 0
\end{align*}
is valid for all $s > 1$ and for all $z_1, z_2 \in \mathbb{R}^N$ with $z_1 \neq z_2$. On account of this, we have that $A$ is strictly monotone.
Therefore, in order to archive $(i)$, we have only to show that $A_{\mathcal{T}}$ is bounded. Given $u \in \Wpzero{\mathcal{T}}$, we put $$M = \min \l \{\frac{1}{\|\nabla u\|_{\mathcal{T}}^{p^- - 1}},  \frac{1}{\|\nabla u\|_{\mathcal{T}}^{r^+ - 1}}\r\}.$$ Then, for $u, v \in \Wpzero{\mathcal{T}}$ we know that
\begin{align*}
& M \ \l \langle A_{\mathcal{T}}(u), \frac{v}{\|\nabla v\|_{\mathcal{T}}} \r \rangle  \\
 \leq &M \, \int_\Omega \l(|\nabla u|^{p(x) - 1} \frac{|\nabla v|}{\|\nabla v\|_{\mathcal{T}}} + \mu_1 |\nabla u|^{q(x) - 1} \frac{|\nabla v|}{\|\nabla v\|_{\mathcal{T}}} + \mu_2 |\nabla u|^{r(x) - 1} \frac{|\nabla v|}{\|\nabla v\|_{\mathcal{T}}} \r) \, \diff x\\
 \leq& \int_\Omega \l |\frac{\nabla u}{\|\nabla u\|_{\mathcal{T}}}\r|^{p(x) -1} \frac{|\nabla v|}{\|\nabla v\|_{\mathcal{T}}} \, \diff x + \int_{\Omega} \mu_1(x)^{\frac{q(x) -1}{q(x)}} \l|\frac{\nabla u}{\|\nabla u\|_{\mathcal{T}}}\r|^{q(x) -1} \mu_1(x)^{\frac{1}{q(x)}} \frac{|\nabla v|}{\|\nabla v\|_{\mathcal{T}}} \, \diff x\\
& + \int_{\Omega} \mu_2(x)^{\frac{r(x) -1}{r(x)}} \l|\frac{\nabla u}{\|\nabla u\|_{\mathcal{T}}}\r|^{r(x) -1} \mu_2(x)^{\frac{1}{r(x)}} \frac{|\nabla v|}{\|\nabla v\|_{\mathcal{T}}} \, \diff x \\
 \leq& \frac{p^+ - 1}{p^-} \int_\Omega \l |\frac{\nabla u}{\|\nabla u\|_{\mathcal{T}}}\r|^{p(x)} \, \diff x + \frac{1}{p^-} \int_{\Omega} \l| \frac{\nabla v}{\|\nabla v\|_{\mathcal{T}}}\r|^{p(x)} \, \diff x \\
 & + \frac{q^+ - 1}{q^-} \int_\Omega \mu_1(x) \l |\frac{\nabla u}{\|\nabla u\|_{\mathcal{T}}}\r|^{q(x)} \, \diff x
+ \frac{1}{q^-} \int_{\Omega} \mu_1(x) \l| \frac{\nabla v}{\|\nabla v\|_{ \mathcal{T}}}\r|^{q(x)} \, \diff x \\
 & +
\frac{r^+ - 1}{r^-} \int_\Omega \mu_2(x) \l |\frac{\nabla u}{\|\nabla u\|_{\mathcal{T}}}\r|^{r(x)} \, \diff x  + \frac{1}{r^-} \int_{\Omega} \mu_2(x) \l| \frac{\nabla v}{\|\nabla v\|_{\mathcal{T}}}\r|^{r(x)} \, \diff x\\
& \ \text{(by using the Young inequality)} \\
 \leq& \frac{r^+ - 1}{p^-} \, \rho_{\mathcal{T}} \l(\frac{\nabla u}{\|\nabla u\|_{\mathcal{T}}}\r) + \frac{1}{p^-} \, \rho_{\mathcal{T}} \l(\frac{\nabla v}{\|\nabla v\|_{\mathcal{T}}}\r)\\
 =& \frac{r^+ - 1}{p^-} + \frac{1}{p^-} = \frac{r^+}{p^-}\\
& \ (\text{on account of Proposition \ref{Prop2} (i)}).
\end{align*}
Consequently, we have
\begin{align*}
	\left\Vert A_{\mathcal{T}}(u)\right\Vert_{*}  & = \sup_{v \in \Wpzero{\mathcal{T}} \setminus \{0\}} \frac{\left\langle A_{\mathcal{T}}(u), v \right\rangle}{\|\nabla v\|_{\mathcal{T}}}\\
	& \leq \frac{r^+}{p^-} \, \max \left\{\|\nabla u\|_{\mathcal{T}}^{p^- - 1}, \|\nabla u\|_{\mathcal{T}}^{r^+ - 1}\right\}
\end{align*}
and this assures that $A_{\mathcal{T}}$ is bounded.

Now, we prove that $(ii)$ holds. For this end, we consider a sequence $\left\{u_n\right\}_{n\in\N} \subseteq \Wpzero{\mathcal{T}}$ such that
\begin{align} \label{S+}
	u_n \weak u \quad \text{in }  \Wpzero{\mathcal{T}} \quad \text{and} \quad \limsup_{n \to + \infty} \, \left\langle A_{\mathcal{T}}\left(u_n\right), u_n - u\right\rangle \leq 0.
\end{align}
We point out that, on account of $u_n \weak u$ in $ \Wpzero{\mathcal{T}}$, we have
\begin{align} \label{Au}
	\lim_{n \to + \infty} \left\langle A_{\mathcal{T}}(u), u_n - u\right\rangle = 0.
\end{align}
Hence, using \eqref{S+}, we also deduce that
\begin{align*}
	\limsup_{n \to \infty} 	\left\langle A_{\mathcal{T}}\left(u_n\right) - A_{\mathcal{T}}(u), u_n - u\right\rangle \leq 0.
\end{align*}
Taking into account that $A_{\mathcal{T}}$ is strictly monotone by $(i)$, we can in addition affirm that
\begin{align*}
0 \leq	\liminf_{n \to + \infty} 	\left\langle A_{\mathcal{T}}\left(u_n\right) - A_{\mathcal{T}}(u), u_n - u\right\rangle \leq \limsup_{n \to + \infty} 	\left\langle A_{\mathcal{T}}\left(u_n\right) - A_{\mathcal{T}}(u), u_n - u\right\rangle \leq 0
\end{align*}
and consequently, we conclude
\begin{align}\label{Au-Av}
	\lim_{n \to + \infty}	\left\langle A_{\mathcal{T}}\left(u_n\right) - A_{\mathcal{T}}(u), u_n - u\right\rangle = 0 = 	\lim_{n \to + \infty} \left\langle A_{\mathcal{T}}(u), u_n - u\right\rangle.
\end{align}
Next, we remark that
\begin{align*}
&	\int_\Omega \l(|\nabla u_n|^{p(x) - 2} \nabla u_n + \mu_1 \left\vert\nabla u_n\right\vert^{q(x) - 2} \nabla u_n + \mu_2 \left\vert\nabla u_n\right\vert^{r(x) - 1} \nabla u_n \r) \cdot \left(\nabla u_n - \nabla u\right)\, \diff x \\
 = & \int_\Omega \left\vert\nabla u_n\right\vert^{p(x)}\, \diff x - \int_\Omega \left\vert\nabla u_n\right\vert^{p(x) - 1} \nabla u\, \diff x\\
& +  \int_\Omega \mu_1(x) \left\vert\nabla u_n\right\vert^{q(x)}\, \diff x - \int_\Omega \mu_1(x) \left\vert\nabla u_n\right\vert^{q(x) - 1} \nabla u\, \diff x\\
& +  \int_\Omega \mu_2(x) \left\vert\nabla u_n\right\vert^{r(x)}\, \diff x - \int_\Omega \mu_2(x) \left\vert\nabla u_n\right\vert^{r(x) - 1} \nabla u\, \diff x\\
 \geq& \int_\Omega \left\vert\nabla u_n\right\vert^{p(x)}\, \diff x - \int_\Omega \left\vert\nabla u_n\right\vert^{p(x) - 1} |\nabla u| \, \diff x\\
& +  \int_\Omega \mu_1(x) \left\vert\nabla u_n\right\vert^{q(x)}\, \diff x - \int_\Omega \mu_1(x) \left\vert\nabla u_n\right\vert^{q(x) - 1} |\nabla u| \, \diff x\\
& +  \int_\Omega \mu_2(x) \left\vert\nabla u_n\right\vert^{r(x)}\, \diff x - \int_\Omega \mu_2(x) \left\vert\nabla u_n\right\vert^{r(x) - 1} |\nabla u| \, \diff x\\
 \geq&  \int_\Omega \left\vert\nabla u_n\right\vert^{p(x)}\, \diff x - \int_\Omega \l(\frac{p(x) - 1}{p(x)}| \nabla u_n |^{p(x)} + \frac{1}{p(x)} |\nabla u|^{p(x)} \r) \, \diff x\\
& + \int_\Omega \mu_1(x) |\nabla u_n|^{q(x)}\, \diff x - \int_\Omega \mu_1(x) \l(\frac{q(x) - 1}{q(x)}\left\vert\nabla u_n\right\vert^{q(x)} + \frac{1}{q(x)} |\nabla u|^{q(x)} \r) \, \diff x\\
& + \int_\Omega \mu_2(x) |\nabla u_n|^{r(x)}\, \diff x - \int_\Omega \mu_2(x) \l(\frac{r(x) - 1}{r(x)}\left\vert\nabla u_n\right\vert^{r(x)} + \frac{1}{r(x)} |\nabla u|^{r(x)} \r) \, \diff x\\
& \ \text{(by using the Young inequality)} \\
 =&  \int_\Omega \frac{1}{p(x)} \left\vert\nabla u_n\right\vert^{p(x)} \, \diff x - \int_\Omega \frac{1}{p(x)} |\nabla u|^{p(x)} \, \diff x\\
& + \int_\Omega \frac{\mu_1(x)}{q(x)} \left\vert\nabla u_n\right\vert^{q(x)} \, \diff x - \int_\Omega \frac{\mu_1(x)}{q(x)} |\nabla u|^{q(x)} \, \diff x\\
& + \int_\Omega \frac{\mu_2(x)}{r(x)} \left\vert\nabla u_n\right\vert^{r(x)} \, \diff x - \int_\Omega \frac{\mu_2(x)}{r(x)} |\nabla u|^{r(x)} \, \diff x.
\end{align*}
On account of this, using \eqref{S+} we deduce that
\begin{align*}
&	\limsup_{n \to + \infty} \int_\Omega \l(\frac{1}{p(x)} \left\vert\nabla u_n\right\vert^{p(x)} + \frac{\mu_1(x)}{q(x)} \left\vert\nabla u_n\right\vert^{q(x)} + \frac{\mu_2(x)}{r(x)} \left\vert\nabla u_n\right\vert^{r(x)}\r) \, \diff x\\
& \leq \int_\Omega \l(\frac{1}{p(x)} |\nabla u|^{p(x)} + \frac{\mu_1(x)}{q(x)} |\nabla u|^{q(x)} + \frac{\mu_2(x)}{r(x)} |\nabla u|^{r(x)}\r) \, \diff x.
\end{align*}
Now, thanks to Fatou's Lemma we know that
\begin{align*}
	&	\liminf_{n \to + \infty} \int_\Omega \l(\frac{1}{p(x)} \left\vert\nabla u_n\right\vert^{p(x)} + \frac{\mu_1(x)}{q(x)} \left\vert\nabla u_n\right\vert^{q(x)} + \frac{\mu_2(x)}{r(x)} \left\vert\nabla u_n\right\vert^{r(x)}\r) \, \diff x\\
	& \geq \int_\Omega \l(\frac{1}{p(x)} |\nabla u|^{p(x)} + \frac{\mu_1(x)}{q(x)} |\nabla u|^{q(x)} + \frac{\mu_2(x)}{r(x)} |\nabla u|^{r(x)}\r) \, \diff x.
\end{align*}
Consequently, we conclude that
\begin{align*}
	&	\lim_{n \to + \infty} \int_\Omega \l(\frac{1}{p(x)} \left\vert\nabla u_n\right\vert^{p(x)} + \frac{\mu_1(x)}{q(x)} \left\vert\nabla u_n\right\vert^{q(x)} + \frac{\mu_2(x)}{r(x)} \left\vert\nabla u_n\right\vert^{r(x)}\r) \, \diff x\\
	& = \int_\Omega \l(\frac{1}{p(x)} |\nabla u|^{p(x)} + \frac{\mu_1(x)}{q(x)} |\nabla u|^{q(x)} + \frac{\mu_2(x)}{r(x)} |\nabla u|^{r(x)}\r) \, \diff x.
\end{align*}
We stress that similar to the proof of Theorem $3.3$ $(ii)$ in Crespo-Blanco-Gasinski-Harjulehto-Winkert \cite{Crespo-Blanco-Gasinski-Harjulehto-Winkert-2022}, we can show that \begin{align*}
	\nabla u_n \to \nabla u \quad \text{in }  L^{p(\cdot)}(\Omega)
\end{align*}
and further, we have that $\left\{\nabla u_n\right\}_{n\in\N}$ converges in measure to $\nabla u$. According of this, we know that the functions on the left-hand side of above equality converge in measure to those on the
right-hand side. Therefore, the converse of Vitali theorem (see Bauer \cite[Lemma $21.6$]{Bauer-2001}) guarantees the uniform integrability of the sequence of functions
\begin{align*}
	\l\{\frac{1}{p(x)} \left\vert\nabla u_n\right\vert^{p(x)} + \frac{\mu_1(x)}{q(x)} \left\vert\nabla u_n\right\vert^{q(x)} + \frac{\mu_2(x)}{r(x)} \left\vert\nabla u_n\right\vert^{r(x)} \r\}_{n \in \mathbb{N}}.
\end{align*}
Now, we stress that
\begin{align*}
	& \left\vert\nabla u_n - \nabla u\right\vert^{p(x)} + \mu_1(x) \left\vert\nabla u_n - \nabla u\right\vert^{q(x)} + \mu_2(x) \left\vert\nabla u_n - \nabla u\right\vert^{r(x)}\\
	 \leq& 2^{r^+ -1}\,  r^+ \l(\frac{1}{p(x)} \left\vert\nabla u_n\right\vert^{p(x)} + \frac{\mu_1(x)}{q(x)}  \left\vert\nabla u_n\right\vert^{q(x)} + \frac{\mu_2(x)}{r(x)}  \left\vert\nabla u_n\right\vert^{r(x)}\r)\\
	& + 2^{r^+ -1} \, r^+ \l(\frac{1}{p(x)} |\nabla u|^{p(x)} + \frac{\mu_1(x)}{q(x)} |\nabla u|^{q(x)} + \frac{\mu_2(x)}{r(x)} |\nabla u|^{r(x)}\r).
\end{align*}
From this, we deduce that the sequence
\begin{align} \label{sequenceun-u}
	\l \{ \left\vert\nabla u_n - \nabla u\right\vert^{p(x)} + \mu_1(x) \left\vert\nabla u_n - \nabla u\right\vert^{q(x)} +
 \mu_2(x) \left\vert\nabla u_n - \nabla u\right\vert^{r(x)}\r\}_{n \in \mathbb{N}}
\end{align}
is uniformly integrable. Then, thanks to the convergence in measure
of $\left\{\nabla u_n\right\}_{n\in\N}$ to $\nabla u$, we have that the sequence \eqref{sequenceun-u} converges in measure to $0$. Now, we can use the Vitali theorem (see Bogachev \cite[Theorem $4.5.4$]{Bogachev-2007}) which gives us
\begin{align*}
	& \lim_{n \to  \infty} \int_{\Omega} \l(\left\vert\nabla u_n - \nabla u\right\vert^{p(x)} + \mu_1(x) \left\vert\nabla u_n - \nabla u\right\vert^{q(x)} + \mu_2(x) \left\vert\nabla u_n - \nabla u\right\vert^{r(x)}\r) \, \diff x \\
	& = \lim_{n \to  \infty} \, \rho_{\mathcal{T}}\left(\nabla u_n - \nabla u\right) = 0.
\end{align*}
Hence, thanks to Proposition \ref{Prop2} $(v)$, we deduce that
\begin{align*}
	\left\Vert u_n - u\right\Vert_{1, \mathcal{T}, 0} \to 0,
\end{align*}
that is, we have that $u_n \to u \, \text{in} \, \Wpzero{\mathcal{T}}$.

Finally, we prove $(iii)$. At first, we point out that $A_{\mathcal{T}}$ is coercive (see Definition \ref{coercive} $(iii)$). We have in fact that
\begin{align*}
\frac{\left\langle A_{\mathcal{T}}(u), u \right\rangle}{\|u\|_{1, \mathcal{T},0}}   =& \int_\Omega \| u \|_{1, \mathcal{T}, 0}^{p(x) - 1} \l(\frac{|\nabla u|}{\|u\|_{1, \mathcal{T}, 0}}\r)^{p(x)} \, \diff x \\
& + \int_\Omega \mu_1(x) \| u \|_{1, \mathcal{T}, 0}^{q(x) - 1} \l(\frac{|\nabla u|}{\|u\|_{1, \mathcal{T}, 0}}\r)^{q(x)} \, \diff x \\
& + \int_\Omega  \mu_2(x) \| u \|_{1, \mathcal{T}, 0}^{r(x) - 1} \l(\frac{|\nabla u|}{\|u\|_{1, \mathcal{T}, 0}}\r)^{r(x)} \, \diff x\\
 \geq& \min \,\{ \| u \|_{1, \mathcal{T}, 0}^{p^- - 1}, \, \| u \|_{1, \mathcal{T}, 0}^{r^+ - 1} \} \ \rho_{\mathcal{T}} \l( \frac{\nabla u}{\|\nabla u\|_{_{\mathcal{T}}}}\r)\\
 =& \min \,\{ \| u \|_{1, \mathcal{T}, 0}^{p^- - 1}, \, \| u \|_{1, \mathcal{T}, 0}^{r^+ - 1} \}\\
& \l(\text{we recall that} \ \text{Proposition \ref{Prop2} (i) gives}\  \rho_{\mathcal{T}} \l( \frac{\nabla u}{\|\nabla u\|_{_{\mathcal{T}}}}\r) = 1\r)\\
& \to + \infty \quad \text{as} \quad \| u \|_{1, \mathcal{T}, 0} \to \infty.
\end{align*}

Next, as $A_{\mathcal{T}}$ is bounded, continuous and strictly monotone by $(i)$, we can apply the Minty-Browder theorem which guarantees that $A_{\mathcal{T}}$ is invertible and its inverse $A_{\mathcal{T}}^{-1}$ is strictly monotone and bounded. Based on this, in order to conclude that $A_{\mathcal{T}}$ is a homeomorphism, we have only to prove that $A_{\mathcal{T}}^{-1}$ is continuous.
To this purpose, we consider a sequence
\begin{align} \label{wn}
\left\{w_n\right\}_{n\in\N} \subseteq \Wpzero{\mathcal{T}}^* \quad \text{such that} \quad w_n \to w \ \text{in} \ \Wpzero{\mathcal{T}}^*.
\end{align}
In addition, we take the sequence $\left\{u_n\right\}_{n\in\N} \subseteq \Wpzero{\mathcal{T}}$ defined by
\begin{align*}
	u_n = A_{\mathcal{T}}^{-1}\left(w_n\right) \quad \text{for all} \ n \in \mathbb{N}
\end{align*}
and further we put $ u = A_{\mathcal{T}}^{-1}(w)$. We note that from \eqref{wn}, using the boundedness of $A_{\mathcal{T}}^{-1}$, we get that
$\left\{u_n\right\}_{n\in\N}$ is bounded in $\Wpzero{\mathcal{T}}$. Hence, we know that there exists a subsequence of $\left\{u_n\right\}_{n\in\N}$ (that we still claim $\left\{u_n\right\}_{n\in\N}$) such that
\begin{align*}
	u_n \weak \bar{u} \quad \text{in} \, \Wpzero{\mathcal{T}}.
\end{align*}
Taking into account this, we have that
\begin{align*}
		& \lim_{n \to + \infty}	\left\langle A_{\mathcal{T}}\left(u_n\right) - A_{\mathcal{T}}(\bar{u}), u_n - \bar{u}\right\rangle\\
		& = \lim_{n \to + \infty} \left\langle w_n - w, u_n - \bar{u}\right\rangle +
		\lim_{n \to + \infty} \left\langle w - A_{\mathcal{T}}(\bar{u}), u_n - \bar{u}\right\rangle  \\
		& = 0.
\end{align*}
Now, the $\left(S_+\right)$-property of operator $A_{\mathcal{T}}$ (see assertion $(ii)$) permits to conclude that
\begin{align*}
	u_n \to \bar{u} \quad \text{in} \, \Wpzero{\mathcal{T}}.
\end{align*}
Then, using the continuity of $A$, we get
\begin{align*}
	A_{\mathcal{T}}(\bar{u}) = \lim_{n \to \infty} A_{\mathcal{T}}\left(u_n\right) = \lim_{n \to \infty} w_n = w = A_{\mathcal{T}}(u).
\end{align*}
Hence, as $A_{\mathcal{T}}$ is injective, we deduce that $\bar{u} = u$ and further we conclude that $(iii)$ holds.
\end{proof}

We emphasize that arguing like in the proofs of Propositions \ref{propertiesI} and \ref{PropertiesA}, taking into account that $\Wp{\mathcal{T}} \hookrightarrow \Lp{\mathcal{T}}$ is a compact embedding (see Proposition \ref{Poincare} $(i)$), we get the following result.

\begin{proposition}
Let hypotheses \eqref{H1} be satisfied.	Let $\bar{A}_{\mathcal{T}}\colon \Wp{\mathcal{T}}\to \Wp{\mathcal{T}}^*$ be the nonlinear operator defined by
	\begin{align*}
		\begin{split}
			\left\langle \bar{A}_{\mathcal{T}}(u),v\right\rangle = \into \left(|\nabla u|^{p(x)-2}\nabla u +\mu_1(x)|\nabla u|^{q(x)-2}\nabla u + \mu_2(x)|\nabla u|^{r(x)-2}\nabla u \right)\cdot\nabla v \,\diff x,
		\end{split}
	\end{align*}
	for all $u,v\in \Wp{\mathcal{T}}$ and let $\bar{I}_{\mathcal{T}}\colon \Wp{\mathcal{T}}\to \mathbb{R} $ be the corresponding energy functional. Then, the following hold:
		\begin{enumerate}
				\item[\textnormal{(i)}] $\bar{I}_{\mathcal{T}}$ is well-defined and of class $C^1$ with $\bar{I}_{\mathcal{T}}^{\prime}(u) = \bar{A}_{\mathcal{T}}(u)$;
		\item[\textnormal{(ii)}] $\bar{A}_{\mathcal{T}}$ is bounded, continuous and strictly monotone;
		\item[\textnormal{(iii)}] $\bar{A}_{\mathcal{T}}$ satisfies the
		$\left(S_+\right)$-property.
	\end{enumerate}
\end{proposition}
\section{Existence and uniqueness results for a multi phase Dirichlet problem} \label{sec5}

This section is devoted to the study of the Dirichlet problem  \eqref{problem}. Our aim is in establishing the existence of at least one nontrivial solution for such a problem. However, we also provide conditions for the uniqueness of solution.

We point out that here the exponents $p(\cdot), q(\cdot), r(\cdot)$ and the weight functions $\mu_1(\cdot), \mu_2(\cdot)$ satisfy assumptions \eqref{H1}. Also, in the analysis of  problem \eqref{problem} we need some known results on the spectrum of the negative $m$-Laplacian with $1< m < + \infty$ and Dirichlet boundary condition given as
\begin{equation}\label{eigenvalue_problem}
	\begin{aligned}
		-\Delta_m u& =\lambda|u|^{m-2}u\quad && \text{in } \Omega,\\
		u
		& = 0  &&\text{on } \partial \Omega.
	\end{aligned}
\end{equation}
Thus for convenience of the reader, we recall that a number $\lambda \in \R$ is an eigenvalue of \eqref{eigenvalue_problem} if problem \eqref{eigenvalue_problem} has a nontrivial solution $u \in W^{1, m}_{0}(\Omega)$. In addition, from L\^{e} \cite{Le-2006} we know that there exists a smallest eigenvalue $\lambda_{1, m}$ of \eqref{eigenvalue_problem} which is positive, isolated and simple. Further, such eigenvalue can be variationally characterized through
\begin{align}\label{characterization-lambda1m}
	\lambda_{1,m} =\inf \left \{ \frac{\|\nabla u\|_{m}^m}{\|u\|_{m}^m}\,:\, u \in W^{1, m}_0(\Omega), u \neq 0 \right \}.
\end{align}
We make use of the first eigenvalue in
the hypotheses on the reaction term $f$. Precisely, we suppose that:

\begin{enumerate}[label=\textnormal{(H$2$)},ref=\textnormal{H$2$}]
	\item\label{H2}
	$f\colon \Omega \times \R \times \R^N \to\R$ is a Carath\'eodory function with $f(\cdot, 0, 0) \neq 0$ satisfying the following conditions:
		\begin{enumerate}[itemsep=0.2cm,label=\textnormal{(\roman*)},ref=\textnormal{\roman*}]
		\item\label{H2i}
		there exist $\gamma_1 \in L^{\frac{m(x)}{m(x)-1}}(\Omega)$ and $k_1, k_2 \geq 0$ such that
		\begin{align*}
			|f(x, t, z)| \leq k_1 |z|^{p(x) \frac{m(x)}{m(x) -1}} + k_2 |t|^{m(x) -1} + \gamma_1(x)
		\end{align*}
	 for a.\,a. $x \in \Omega$, for all $t \in \R$ and for all $ z \in \R^N$, where $m \in C(\close)$ is such that $1 < m(x) < p^*(x) : = \frac{N p(x)}{N - p(x)}$;
		\item\label{H2ii}
		there exist $\gamma_2 \in L^1(\Omega)$ and $k_3, k_4 \geq 0$ such that
		\begin{align*}
			f(x, t, z) \, t \leq k_3 |z|^{p(x)} + k_4 |t|^{p^-} + \gamma_2(x)
		\end{align*}
	for a.\,a. $x \in \Omega$, for all $t \in \R$ and for all $ z \in \R^N$. Also, we suppose
		\begin{align*}
			1 - k_3 - k_4 \, \lambda_{1, p^-}^{-1} > 0,
		\end{align*}
	being $\lambda_{1, p^-}$ the first eigenvalue of negative $p^-$-Laplacian with Dirichlet boundary condition (see \eqref{characterization-lambda1m}).
	\end{enumerate}
\end{enumerate}

\begin{definition}\label{weaksolution}
	We claim $u \in \Wpzero{\mathcal{T}}$ a weak solution of problem \eqref{problem} if for all $h \in \Wpzero{\mathcal{T}}$ we have
\begin{equation} \label{weakeq}
\int_\Omega \l (|\nabla u|^{p(x)-2} \nabla u + \mu_1(x)|\nabla u|^{q(x)-2}\nabla u + \mu_2(x)|\nabla u|^{r(x)-2}\nabla u \r) \cdot \nabla h \, \diff x
\end{equation}
\begin{equation*}
 =  \int_\Omega f(x,u, \nabla u) \ h \, \diff x.
\end{equation*}

\end{definition}

\begin{remark}
We stress that thanks to the embedding results in Proposition \ref{embedding} and to hypothesis \eqref{H2}\eqref{H2i}, we have that the weak solutions of problem \eqref{problem} are well-defined.
\end{remark}
Now, we are ready to give our existence result.
\begin{theorem} \label{existence}
	Let hypotheses \eqref{H1} and \eqref{H2} be satisfied. Then, problem \eqref{problem} admits at least one
	nontrivial weak solution $u \in \Wpzero{\mathcal{T}}$.
\end{theorem}
\begin{proof}
	
	Let $N^*_{f}\colon \Wpzero{\mathcal{T}} \subset \Lp{m(\cdot)} \to \Lp{m^{\prime}(\cdot)}$ be the Nemytskij operator corresponding to the Carath\'eodory function $f$, that is,
	\begin{align*}
		N^*_{f}(u)(\cdot)= f\l(\cdot, u(\cdot), \nabla u(\cdot)\r) \quad \text{for all} \ u \in \Wpzero{\mathcal{T}}.
	\end{align*}
Thanks to hypothesis \eqref{H2}\eqref{H2i} we have that $N^\ast_{f}(\cdot)$ is well-defined, bounded and continuous, see Motreanu-Motreanu-Papageorgiou \cite{Motreanu-Motreanu-Papageorgiou-2014}. Hence, since the embedding $i^*\colon \Lp{m^{\prime}(\cdot)} \to \Wpzero{\mathcal{T}}^* $ is continuous (see Gasinski-Papageorgiou \cite[Lemma $2.2.27$, p. $141$]{Gasinski-Papageorgiou-2006}), we deduce that also the operator $N_{f}\colon \Wpzero{\mathcal{T}} \to \Wpzero{\mathcal{T}}^*$ defined by $N_f=i^* \circ N^*_f$ is bounded and continuous.

Now, we consider the operator $\mathcal{A}_{\mathcal{T}}\colon \Wpzero{\mathcal{T}} \to \Wpzero{\mathcal{T}}^*$ given by
	\begin{align*}
		\mathcal{A}_{\mathcal{T}}(u) = A_{\mathcal{T}}(u) - N_{f}(u) \quad \text{for all } u \in \Wpzero{\mathcal{T}}.
	\end{align*}
	Since $N_{f}$ is bounded and continuous, on account of Theorem \ref{PropertiesA}, we have that $\mathcal{A}_{\mathcal{T}}$ is bounded and continuous. We stress that if $\mathcal{A}_{\mathcal{T}}$ is in addition pseudomonotone and coercive then, thanks to Theorem \ref{theorem_pseudomonotone}, we have that $\mathcal{A}_{\mathcal{T}}$ is surjective. This in particular guarantees the existence of a function $u \in \Wpzero{\mathcal{T}}$ such that $\mathcal{A}_{\mathcal{T}}(u) = 0$. Now, according to the definition of the operator $\mathcal{A}_{\mathcal{T}}$ and as $f(\cdot, 0, 0) \neq 0$ (see hypothesis \eqref{H2}), we have that $u$ is a nontrivial weak solution of problem \eqref{problem}, hence the claim is proved.
	
In order to conclude the proof we at first show that $\mathcal{A}_{\mathcal{T}}$ is pseudomonotone (in the sense of Remark \ref{RemarkPseudomonotone}). We consider
	a sequence $\{u_n\}_{n\in\N} \subseteq \Wpzero{\mathcal{T}}$ such that
	\begin{align} \label{un}
u_n \weak u \ \text{in} \ \Wpzero{\mathcal{T}}  \quad \text{and} \quad \limsup_{n\to +\infty}
\left\langle \mathcal{A}_{\mathcal{T}}\left(u_n\right),u_n-u\right\rangle \leq 0.
	\end{align}
Now, according to the weak convergence of $\left\{u_n\right\}_{n\in\N}$ in $\Wpzero{\mathcal{T}}$, we have that $\left\{u_n\right\}_{n\in\N}$ is bounded in its norm. This implies that $\left\{N_f^*\left(u_n\right)\right\}_{n\in\N}$ is bounded too. Then, using this fact along with H\"older's inequality and the compact embedding $\Wpzero{\mathcal{T}} \hookrightarrow \Lp{m(\cdot)}$ (see Proposition \ref{embedding} $(iii)$), we obtain
\begin{align}\label{convergence_Nf}
	\begin{split}
		\l|\into f\l(x,u_n,\nabla u_n\r)(u_n-u)\,\diff x \r| & \leq
	2 \l\| N^*_f\left(u_n\right)\r\|_{\frac{m(\cdot) - 1}{m(\cdot)}} \l\|u-u_n\r\|_{m(\cdot)} \\
		& \leq 2 \sup_{n \in \N} \l\|N^*_f\left(u_n\right)\r\|_{\frac{m(\cdot) -
				1}{m(\cdot)}} \l\|u-u_n\r\|_{m(\cdot)}\\
			& \to 0\quad\text{as }n\to
		\infty.
	\end{split}
\end{align}
Passings to the limit in \eqref{weakeq} (where we replace $u$ by $u_n$ and $h$ by $u_n-u$) and using \eqref{convergence_Nf}, we further get
	\begin{align} \label{limsupA}
	\limsup_{n \to +\infty}\,\l\langle A_{\mathcal{T}}\left(u_n\right), u_n-u\r\rangle =
	\limsup_{n \to +\infty}\,\l\langle \mathcal{A}_{\mathcal{T}}\left(u_n\right), u_n-u\r\rangle \leq 0.
\end{align}
We recall that Theorem \ref{PropertiesA} assures that $A_{\mathcal{T}}$ satisfies the $\Ss$-property. On account of this, by \eqref{un} and \eqref{limsupA}, we deduce that $u_n\to u$ in $\Wpzero{\mathcal{T}}$. Now, using the continuity of $\mathcal{A}_{\mathcal{T}}$, we have that $\mathcal{A}_{\mathcal{T}}\left(u_n\right)\to \mathcal{A}_{\mathcal{T}}(u)$ in $\Wpzero{\mathcal{T}}^*$. This guarantees, with a view to Remark \ref{RemarkPseudomonotone}, that $\mathcal{A}_{\mathcal{T}}$ is pseudomonotone.

Next, we prove that $\mathcal{A}_{\mathcal{T}}$ is coercive. We recall that from the variational characterization of the first eigenvalue of negative Dirichlet $p^-$-Laplacian (see \eqref{characterization-lambda1m}), we know that
\begin{align}\label{estimate}
	\into |u|^{p^-}\,\diff x \leq \lambda_{1, p^-}^{-1} \into |\nabla u|^{p^-}\,\diff x\quad\text{for all }u \in \Wpzero{\mathcal{T}} \subset W^{1, p^-}_0(\Omega).
\end{align}
Hence, for $u \in \Wpzero{\mathcal{T}}$ we get
\begin{align*}
	&\left\langle \mathcal{A}_{\mathcal{T}}(u), u\right\rangle\\
	&= \into \l(|\nabla u|^{p(x)} +\mu_1 |\nabla u|^{q(x)} + \mu_2 |\nabla u|^{r(x)}\r) \,\diff x - \into f(x,u,\nabla u)u\,\diff x\\
	& \geq \rho_{\mathcal{T}}(\nabla u) - k_3 \int_\Omega |\nabla u|^{p(x)} \, \diff x - k_4 \int_\Omega |u|^{p^-} \, \diff x - \left\Vert\gamma_2(x)\right\Vert_1 \\
	& \  (\text{by using hypothesis} \ \eqref{H2}\eqref{H2ii})\\
	& \geq \rho_{\mathcal{T}}(\nabla u) - k_3 \int_\Omega |\nabla u|^{p(x)} \, \diff x - k_4 \, \lambda_{1, p^-}^{- 1} \int_\Omega |\nabla u|^{p^-} \, \diff x - \left\Vert\gamma_2(x)\right\Vert_1 \\
	& \ (\text{by using } \ \eqref{estimate})\\
	&\geq \left(1 - k_3 - k_4 \, \lambda_{1, p^-}^{- 1}\right) \, \rho_{\mathcal{T}}(\nabla u) - k_4 \, \lambda_{1, p^-}^{- 1} |\Omega| - \left\Vert\gamma_2(x)\right\Vert_1  \\
	& \geq \left(1 - k_3 - k_4 \, \lambda_{1, p^-}^{- 1}\right) \, \min\left\{\left\Vert\nabla u\right\Vert_{\mathcal{T}}^{p^-}, \|\nabla u\|_{\mathcal{T}}^{r^+}\right\} - k_4 \, \lambda_{1, p^-}^{- 1} |\Omega| - \left\Vert\gamma_2(x)\right\Vert_1\\
	& \ (\text{by using} \ (iii) \ \text{and} \ (iv) \ \text{of Proposition} \ \ref{Prop2}).
\end{align*}
Now, being $ 1 - k_3 - k_4 \, \lambda_{1, p^-}^{- 1} > 0$ (see hypothesis \eqref{H2}\eqref{H2ii}), we know that
$$\frac{1}{\|u\|_{1, \mathcal{T}, 0}}\l[\left(1 - k_3 - k_4 \, \lambda_{1, p^-}^{- 1}\right) \, \min\left\{\|\nabla u\|_{\mathcal{T}}^{p^-}, \|\nabla u\|_{\mathcal{T}}^{r^+}\right\} - k_4 \, \lambda_{1, p^-}^{- 1} |\Omega| - \left\Vert\gamma_2(x)\right\Vert_1\r]$$ goes to $+ \infty$ for $\|u\|_{1, \mathcal{T}, 0} \to + \infty$ (we recall that $\|u\|_{1, \mathcal{T}, 0} := \|\nabla u\|_{ \mathcal{T}}$ for all $u \in \Wpzero{\mathcal{T}}$). Consequently, with a view to Definition \ref{coercive}, we conclude that $\mathcal{A}_{\mathcal{T}}(\cdot)$ is coercive. Then, we can finally apply Theorem \ref{theorem_pseudomonotone} which gives us the existence of a nontrivial weak solution of problem \eqref{problem}.

\end{proof}

In what follows, we suppose that the nonlinearity $f\colon \Omega \times \R \times \R^N \to\R$ satisfies the following additional conditions:
\begin{enumerate}[label=\textnormal{(H$3$)},ref=\textnormal{H$3$}]
	\item\label{H3}
	\begin{enumerate}[itemsep=0.2cm,label=\textnormal{(\roman*)},ref=\textnormal{\roman*}]
		\item\label{H3i}
		there exists $k_5 \geq 0$ such that
		\begin{align*}
			(f(x, t, z) - f(x, s, z)) (t - s) \leq k_5 \, |t - s|^2
		\end{align*}
		for a.\,a. $x \in \Omega$, for all $t, s \in \R$ and for all $ z \in \R^N$;
		\item\label{H3ii}
		there exist $\gamma_3 \in L^{m^{\prime}(\cdot)}(\Omega)$, with $m \in C(\close)$ such that $1 < m(x) < p^*(x)$ for all $x \in \close$, and $k_6 \geq 0$ so that the map
		\begin{align*}
		z \to	f(x, t, z) - \gamma_3(x)
		\end{align*}
	is linear for a.\,a. $x \in \Omega$, for all $t \in \R$ and for all $ z \in \R^N$, and further
		\begin{align*}
			\left\vert f(x, t, z) - \gamma_3(x)\right\vert \leq k_6 |z|
		\end{align*}
		for a.\,a. $x \in \Omega$, for all $t \in \R$ and for all $ z \in \R^N$. Also, we suppose
		\begin{align} \label{k5k6}
			k_5 \lambda_{1,2}^{-1} + k_6 \, \lambda_{1, 2}^{- \frac{1}{2}} < 1,
		\end{align}
		being $\lambda_{1, 2}$ the first eigenvalue of the negative Laplacian with Dirichlet boundary condition (see \eqref{characterization-lambda1m}).
	\end{enumerate}
\end{enumerate}
Then, the following uniqueness result holds as well.
\begin{theorem} \label{unicity1}
	Let hypotheses \eqref{H1}, \eqref{H2} and \eqref{H3} be satisfied. If $p(x) = 2$ for all $x \in \close$, then problem \eqref{problem} admits a unique weak solution.
\end{theorem}
\begin{proof}
	Let $u, w \in \Wpzero{\mathcal{T}}$ be two weak solutions to problem \eqref{problem} with $p(x) = 2$. Our aim is to prove that $u = w$. To this purpose, we point out that as $u$ and $w$ are weak solutions to problem \eqref{problem} with $p(x) = 2$, then we have
	\begin{align*}
		\int_\Omega \l (\nabla u + \mu_1(x)|\nabla u|^{q(x)-2} \nabla u + \mu_2(x)|\nabla u|^{r(x)-2}\nabla u \r) \cdot \nabla h \, \diff x 	=  \int_\Omega f(x,u, \nabla u) \ h \, \diff x
\end{align*}
	and
		\begin{align*}
		\int_\Omega \l (\nabla w + \mu_1(x)|\nabla w|^{q(x)-2}\nabla w + \mu_2(x)|\nabla w|^{r(x)-2}\nabla w \r) \cdot \nabla h \, \diff x 		=  \int_\Omega f(x, w, \nabla w) \ h \, \diff x
	\end{align*}
for all $h \in \Wpzero{\mathcal{T}}$. Now, if we choose $h = u - w$, from the previous equalities we deduce that
\begin{align*}
		& \int_\Omega |\nabla (u - w)|^2 \, \diff x + \int_\Omega \mu_1(x) \left(|\nabla u|^{q(x)-2} \nabla u - |\nabla w|^{q(x)-2}\nabla w\right) \cdot \nabla (u - w) \, \diff x\\
		& + \int_\Omega \mu_2(x) \left(|\nabla u|^{r(x)-2}\nabla u - |\nabla w|^{r(x)-2} \nabla w\right) \cdot \nabla (u - w) \, \diff x \\
	 =& \int_\Omega (f(x, u, \nabla u) - f(x, w, \nabla u)) (u - w) \, \diff x\\
	& + \int_\Omega  (f(x, w, \nabla u) - f(x, w, \nabla w)) (u - w) \, \diff x.
\end{align*}
Taking into account that both
\begin{align*}
	\int_\Omega \mu_1(x) \left(|\nabla u|^{q(x)-2} \nabla u - |\nabla w|^{q(x)-2}\nabla w\right) \cdot \nabla (u - w) \, \diff x
\end{align*}
and
\begin{align*}
	\int_\Omega \mu_2(x) \left(|\nabla u|^{r(x)-2}\nabla u - |\nabla w|^{r(x)-2} \nabla w\right)  \cdot \nabla (u - w) \, \diff x
\end{align*}
 are nonnegative, we have
 \begin{align*}
 	 \| \nabla (u - w) \|_2^2 \, = \int_\Omega |\nabla (u - w)|^2 \, \diff x & \leq  \int_\Omega (f(x, u, \nabla u) - f(x, w, \nabla u)) (u - w) \, \diff x\\
&	+ \int_\Omega  (f(x, w, \nabla u) - f(x, w, \nabla w)) (u - w) \, \diff x.
 \end{align*}
Now, we stress that \eqref{H3}\eqref{H3i} guarantees that
\begin{align} \label{eqh3i}
\int_\Omega (f(x, u, \nabla u) - f(x, w, \nabla u)) (u - w) \, \diff x \leq k_5 \| u - w \|_2^2
\end{align}
for some $k_5 \geq 0$ and for a.a. $x \in \Omega$. Also, using hypothesis \eqref{H3}\eqref{H3ii} we get the inequality
\begin{align} \label{eqh3ii}
	 &\int_\Omega  (f(x, w, \nabla u) - f(x, w, \nabla w)) (u - w) \, \diff x
	\end{align}
\begin{align*}
	 &\leq \,k_6 \int_\Omega |u - w| \, |\nabla (u - w)| \, \diff x
\end{align*}
for some $k_6 \geq 0$ and for a.a. $x \in \Omega$.

Consequently, thanks to \eqref{eqh3i} and \eqref{eqh3ii} we can write that
\begin{align*}
	\| \nabla (u - w) \|_2^2 \,  & \leq
	k_5 \| u - w \|_2^2 + k_6 \int_\Omega |u - w| \, |\nabla (u - w)| \, \diff x\\
	& \leq	k_5 \| u - w \|_2^2 + k_6 \,\|u - w\|_2 \, \|\nabla (u - w)\|_2\\
	&  \ (\text{by the H\"{o}lder inequality})\\
	& \leq \left(k_5 \lambda_{1,2}^{-1} + k_6 \, \lambda_{1, 2}^{- \frac{1}{2}} \right) \, \|\nabla (u - w)\|_2^2\\
	& \ (\text{by using} \ \eqref{characterization-lambda1m}).
\end{align*}
Hence, taking into account that $\left(k_5 \lambda_{1,2}^{-1} + k_6 \, \lambda_{1, 2}^{- 1/2}\right) < 1$ (see \eqref{k5k6}), we finally conclude that $u = w$.
\end{proof}
Now, we suppose $\mu_1(\cdot)  \in \Linf$ such that
\begin{align} \label{infmu1}
	\inf \mu_1 := \inf_{x \in \close} \mu_1(x) > 0.
\end{align}
In addition, we replace condition \eqref{k5k6} of hypothesis \eqref{H3} by
\begin{align} \label{k5k6new}
	k_5 \lambda_{1,2}^{-1} + k_6 \, \lambda_{1, 2}^{- \frac{1}{2}} < \inf \mu_1.
\end{align} Then, we have the following uniqueness result.
\begin{theorem} \label{unicity2}
	Let hypotheses \eqref{H1}, \eqref{H2} and \eqref{H3} (where we replace \eqref{k5k6} by \eqref{k5k6new}) be satisfied. In addition, assume that $q(x) = 2$ for all $x \in \close$ and \eqref{infmu1} hold as well. Then, problem \eqref{problem} admits a unique weak solution.
\end{theorem}
\begin{proof}
Let	$u, w \in \Wpzero{\mathcal{T}}$ be two weak solutions to problem \eqref{problem} with $q(x) = 2$. Then, we know that
	\begin{align*}
		\int_\Omega \l (|\nabla u|^{p(x)-2} \nabla u + \mu_1(x) \nabla u \ + \mu_2(x)|\nabla u|^{r(x)-2}\nabla u \r) \cdot \nabla h \, \diff x
		=  \int_\Omega f(x,u, \nabla u) \  h \, \diff x
	\end{align*}
	and
	\begin{align*}
		\int_\Omega \l (|\nabla w|^{p(x)-2}\nabla w + \mu_1(x) \nabla w + \mu_2(x)|\nabla w|^{r(x)-2}\nabla w \r) \cdot \nabla h \, \diff x
		=  \int_\Omega f(x, w, \nabla w) \  h \, \diff x
	\end{align*}
	for all $h \in \Wpzero{\mathcal{T}}$. Now, if we take $h = u - w$, using the previous equalities, we get
	\begin{align*}
		& \int_\Omega \left(|\nabla u|^{p(x)-2} \nabla u - |\nabla w|^{p(x)-2}\nabla w\right)  \cdot \nabla (u - w) \, \diff x + \int_\Omega \mu_1(x) |\nabla (u - w)|^2 \, \diff x\\
		& + \int_\Omega \mu_2(x) \left(|\nabla u|^{r(x)-2}\nabla u - |\nabla w|^{r(x)-2} \nabla w\right)  \cdot \nabla (u - w)  \, \diff x \\
		 =& \int_\Omega (f(x, u, \nabla u) - f(x, w, \nabla u)) (u - w) \, \diff x\\
		& + \int_\Omega  (f(x, w, \nabla u) - f(x, w, \nabla w)) (u - w) \, \diff x.
	\end{align*}
	Since both
	\begin{align*}
		\int_\Omega \left(|\nabla u|^{p(x)-2} \nabla u - |\nabla w|^{p(x)-2}\nabla w\right)  \cdot \nabla (u - w) \, \diff x
	\end{align*}
	and
	\begin{align*}
		\int_\Omega \mu_2(x) \left(|\nabla u|^{r(x)-2}\nabla u - |\nabla w|^{r(x)-2} \nabla w\right)  \cdot \nabla (u - w) \, \diff x
	\end{align*}
	are nonnegative, further we have that
	\begin{align*}
	\inf \mu_1 \, \| \nabla (u - w) \|_2^2 \, \leq& \int_\Omega \mu_1(x) |\nabla (u - w)|^2 \, \diff x  \\
		 \leq&  \int_\Omega (f(x, u, \nabla u) - f(x, w, \nabla u)) (u - w) \, \diff x\\
		&	+ \int_\Omega  (f(x, w, \nabla u) - f(x, w, \nabla w)) (u - w) \, \diff x.
	\end{align*}
Following the similar arguments to the ones in the proof of Theorem \ref{unicity1}, then we obtain
\begin{align*}
\inf \mu_1  \, \| \nabla (u - w) \|_2^2	\,\leq \,\left(k_5 \lambda_{1,2}^{-1} + k_6 \, \lambda_{1, 2}^{- \frac{1}{2}} \right) \, \|\nabla (u - w)\|_2^2.
\end{align*}
Hence, according to \eqref{k5k6new}, we conclude that $u = w$, that is, the claim holds.
\end{proof}

Finally, we emphasize that if $\mu_2(\cdot)  \in \Linf$ is such that
\begin{align} \label{infmu2}
	\inf \mu_2 := \inf_{x \in \close} \mu_2(x) > 0
\end{align}
and in addition we replace condition \eqref{k5k6} of hypothesis \eqref{H3} by
\begin{align} \label{k5k6new2}
	k_5 \lambda_{1,2}^{-1} + k_6 \, \lambda_{1, 2}^{- \frac{1}{2}} < \inf \mu_2,
\end{align} then, arguing like in the proof of Theorem \ref{unicity2}, we can deduce the following result.
\begin{theorem} \label{unicity3}
	Let hypotheses \eqref{H1}, \eqref{H2} and \eqref{H3} (where we replace \eqref{k5k6} by \eqref{k5k6new2}) be satisfied. In addition, assume that $r(x) = 2$ for all $x \in \close$ and \eqref{infmu2} hold as well. Then, problem \eqref{problem} admits a unique weak solution.
\end{theorem}


\section{Regularity results for minimizers and higher integrability} \label{sec6}

Given a bounded domain $\Omega \subseteq \mathbb{R}^N$, let now $p(\cdot), q(\cdot), r(\cdot)$ and $\mu_1(\cdot), \,\mu_2(\cdot)$ be functions satisfying the following conditions:
\begin{enumerate}[label=\textnormal{($\text{H}^{\prime}$)},ref=\textnormal{$\text{H}^{\prime}$}]
	\item\label{I}
	$p, q, r \in C^{0, \sigma}(\Omega)$, with $\sigma \in (0, 1]$, are such that $$1 < p_0 \leq p(x) \leq q(x) \leq r(x)$$ for all $x\in\Omega$, where $p_0$ is a fixed constant strictly larger than one, and $\mu_1, \mu_2 \in \Linf$ are such that $$ \mu_1(x)\geq 0 \quad \text{and} \quad \mu_2(x)  \geq 0$$
	for all $x\in\Omega$. In addition, we suppose
	 \begin{align} \label{ineqHprime}
		\sup_{x \in \Omega}\frac{q(x)}{p(x)} < 1 + \frac{\sigma}{N} \quad \text{and} \quad  \sup_{x \in \Omega}\frac{r(x)}{p(x)} < 1 + \frac{\sigma}{N}.
	\end{align}
\end{enumerate}
We denote the H$\ddot{\text{o}}$lder constants of $p(\cdot), q(\cdot),  r(\cdot)$ by $L_p, L_q, L_r$, respectively, that is, we here put
\begin{align} \label{Holderconstant}
	L_p := \sup_{x, y \in \Omega \, \text{and} \, x \neq y} \frac{|p(x) - p(y)|}{|x - y|^{\sigma}} < + \infty,\nonumber\\
	L_q := \sup_{x, y \in \Omega \, \text{and} \, x \neq y} \frac{|q(x) - q(y)|}{|x - y|^{\sigma}} < + \infty, \nonumber\\
	L_r := \sup_{x, y \in \Omega \, \text{and} \, x \neq y} \frac{|r(x) - r(y)|}{|x - y|^{\sigma}} < + \infty.
\end{align}
By $B_R$ we mean a ball contained in $\Omega$ of radius $R > 0$, and by $B_R(y)$ we denote a ball contained in $\Omega$ of center $y$ and radius $R > 0$. Also, for  given $B_R \subset \Omega$ we put
\begin{align*}
	p_{R} = \inf_{x \in B_R} p(x) \quad  \text{and} \quad  p^{R} = \sup_{x \in B_R} p(x),\\
	q_{R} = \inf_{x \in B_R} q(x) \quad  \text{and} \quad  q^{R} = \sup_{x \in B_R} q(x),\\
	r_{R} = \inf_{x \in B_R} r(x) \quad  \text{and} \quad  r^{R} = \sup_{x \in B_R} r(x).
\end{align*}
Further, given a measurable function $u: \Omega \rightarrow \R$ we write $u_{B_{R}}$ for the integral average of $u$ on $B_{R}$, namely
\begin{align} \label{media}
	u_{B_{R}} := -\hspace{-0,38cm}\int_{B_R} u \, \diff x = \frac{1}{\left\vert B_{R}\right\vert} \ \int_{B_{R}} u \, \diff x,
\end{align}
being $\left\vert B_{R}\right\vert$ the Lebesgue measure of $B_R$.

Finally, we stress that from now on we use $C$ to denote a generic constant, which may change its value from line to line, but does not depend on the crucial quantities. This is in order to streamline the notation.
When we need to specify a constant, we use letters with index.

Now, under hypothesis \eqref{I}, we consider the integral functional $\mathcal{F}_{\mathcal{T}}(u; \Omega)$ defined by
\begin{align} \label{intfunctionalT}
	\mathcal{F}_{\mathcal{T}}(u; \Omega) := \int_{\Omega} \mathcal{T}(x, |\nabla u|) \, \diff x
\end{align}
for all $u \in W^{1, 1}(\Omega)$, where $\mathcal{T}$
is the nonlinear function introduced in Section \ref{sec3}.

\begin{definition} \label{minimizer}
	We claim $u \in W^{1, p(x)}(\Omega)$ a local minimizer of the functional $\mathcal{F}_{\mathcal{T}}$ if $$\mathcal{T}(x, |\nabla u|) \in L^1(\Omega),$$
	and in addition the minimality condition
	\begin{align*}
		\mathcal{F}_{\mathcal{T}}(u; \, \text{supp}\, (u-w)) \leq \mathcal{F}_{\mathcal{T}}(w; \, \text{supp} \, (u-w))
	\end{align*}
	is satisfied for each $w \in W^{1, p(x)}(\Omega)$ such that $supp \, (u -w) \subset \Omega$.
\end{definition}

Our aim is in establishing local regularity results for the minimizers of the integral functional $\mathcal{F}_{\mathcal{T}}$ introduced in \eqref{intfunctionalT}. We point out that an useful tool for this purpose is the following remarkable result of Colombo-Mingione, see \cite[Theorem 1.6]{Colombo-Mingione-2015-1}.

\begin{theorem} [Sobolev-Poincar\'{e} inequality \cite{Colombo-Mingione-2015-1}] \label{CM}
	 Let $\tilde{p}, \, \tilde{q} \in \R$ be such that $1 < \tilde{p} < \tilde{q} $ and let $\nu \in C^{0, \sigma_1}(\Omega)$, with $\sigma_1  \in (0, 1]$, be a function such that $\nu(x) \geq 0$ for all $x \in \Omega$. Also, let
	  $H\colon \Omega \times [0, + \infty)  \to [0, +\infty)$ be defined by
	 \begin{align*}
	 	H(x, t) = t^{\tilde{p}} +\nu(x) \,t^{\tilde{q}}
	 \end{align*}
	 for all $x \in \Omega$ and all $t \geq 0$. If
	 \begin{align*}
	 \frac{\tilde{q}}{\tilde{p}} < 1 + \frac{\sigma_1}{N},
	 \end{align*}
	 then there exist a constant $c_1$ depending by $N, \tilde{p}, \tilde{q}$ and exponents $d_1 > 1 > d$, depending by $N, p, q, \sigma_1$, such that
	\begin{equation}\label{ECM}
		\l(-\hspace{-0,36cm} \int_{B_R} \l[H \l(x, \frac{|w-w_{B_R}|}{R}\r)\r]^{d_1} \, \diff x \r)^{\frac{1}{d_1}} \leq c_1 \left(	-\hspace{-0,36cm} \int_{B_R} H (x, |\nabla w|)^{d} \, \diff x \right)^{\frac{1}{d}}
	\end{equation}
	holds for all $w \in W^{1, \tilde{p}}(B_R)$ and for all $B_R \subset \Omega$ with $R \leq 1$.
\end{theorem}

From Theorem \ref{CM} it is in fact possible to derive the following result (see also \cite[Lemma 2]{DeFilippis-Oh-2019}).

\begin{theorem} \label{DFO}
Let $\tilde{p}, \, \tilde{q}, \, \tilde{r} \in \R$ be such that $1 < \tilde{p} < \tilde{q} < \tilde{r} $ and let $\nu_1, \nu_2 \in C^{0, \sigma_2}(\Omega)$, with $\sigma_2  \in (0, 1]$, be such that $\nu_1(x), \, \nu_2(x) \geq 0$ for all $x \in \Omega$.

Also, let $T: \Omega \times [0, + \infty[ \rightarrow [0, + \infty[$ be defined by
\begin{align*}
	T(x, t) = t^{\tilde{p}} +\nu_1(x) \, t^{\tilde{q}} + \nu_2 \, t^{\tilde{r}}
\end{align*}
for all $x \in \Omega$ and all $t \geq 0$. If
\begin{align*}
	\frac{\tilde{q}}{\tilde{p}} < 1 + \frac{\sigma_2}{N} \quad \text{and} \quad \frac{\tilde{r}}{\tilde{p}} < 1 + \frac{\sigma_2}{N},
\end{align*}
	then there exist a constant $c_2$ depending on $N, \tilde{p}, \tilde{q}, \tilde{r}$ and $0 < d < 1$, depending on $N, \tilde{p}, \tilde{q}, \tilde{r}, \sigma_2$, such that  the inequality
	\begin{equation}\label{ECM1}
		\l(-\hspace{-0,36cm} \int_{B_R} \l[T \l(x, \frac{\left\vert w-w_{B_R}\right\vert}{R}\r)\r] \, \diff x \r) \leq c_2 \left(	-\hspace{-0,36cm} \int_{B_R} T (x, |\nabla w|)^{d} \, \diff x \right)^{\frac{1}{d}}
	\end{equation}
	holds for all $w \in W^{1, \tilde{p}}(B_R)$ and for all $B_R \subset \Omega$ with $R \leq 1$.
\end{theorem}
Further, under the same hypotheses of Theorem \ref{DFO}, the following remark holds as well.
\begin{remark} [Remark 2 \cite{DeFilippis-Oh-2019}] \label{RemarkDeFilippisOh}
Let the hypotheses of Theorem \ref{DFO} be satisfied. Then an inequality of type \eqref{ECM1} holds for a function $w \in W^{1, T}(\Omega)$ such that $w \equiv 0$ on a set $E$ with measure $|E|$ satisfying the condition $|E| \geq \gamma \left\vert B_R\right\vert$ for some $0 < \gamma < 1$. Precisely, we have
	\begin{align} \label{SobPoiEqBIS}
	-\hspace{-0,38cm}	\int_{B_{R}} T \l(x, \frac{|w|}{R}\r)\,  \diff x &
		\leq C \l(-\hspace{-0,38cm} \int_{B_R} T(x, |\nabla w|)^{d}\, \diff x \r)^{\frac{1}{d}}
	\end{align}
	where $d \in (0, 1)$ is the same as the corresponding one appearing in \eqref{ECM1} and $C > 0$ depending on $\tilde{p}, \tilde{q}, \tilde{r}, N, \gamma$.
\end{remark}
Finally, we make use of the following well known result of Giaquinta and Modica, for whose proof we refer to Giusti \cite{Giusti-2003}.

\begin{theorem} \label{Giaquinata}
	Let $g_1 \in  L^1(\Omega)$ and suppose that for some $\vartheta \in (0, 1)$ the estimate
	\begin{align*}
			-\hspace{-0,36cm} \int_{B_{R/2}} g_1 \, \diff x \leq c_2	\l(-\hspace{-0,36cm} \int_{B_R} g_{1}^{\vartheta} \,\diff x \r)^{\frac{1}{\vartheta}} + c_2  -\hspace{-0,36cm} \int_{B_R} g_2  \, \diff x
	\end{align*}
	holds for all $B_R \Subset \Omega$ and for some positive constant $c_2$. If $g_2 \in L^m\left(B_R\right)$ for some $m > 1$, then there exists $1 < \delta < m$ such that $g_1 \in  L^{\delta}_{\rm loc}(\Omega)$ and
	\begin{align*}
		-\hspace{-0,36cm} \int_{B_{R/2}} g_{1}^\delta \, \diff x \leq c_2 \l(-\hspace{-0,36cm} \int_{B_{R}} g_1 \, \diff x  \r)^\delta + c_2 -\hspace{-0,36cm} \int_{B_R} g_{2}^\delta \, \diff x.
	\end{align*}
	
\end{theorem}
Here, we start by giving the following lemma which will be used in order to establish our regularity results.

\begin{lemma} \label{Romegaprime}
	Let hypothesis \eqref{I} be satisfied. In addition, let $0 < R_0 \leq 1$  satisfy the following condition
	\begin{align} \label{ROmegaprime}
		R_{0}^{\sigma} \leq \frac{p_0}{2^{1 + \sigma} L_r} \l(1 + \frac{\sigma}{N} - \sup_{x \in \Omega}\frac{r(x)}{p(x)} \r),
	\end{align}
	where $\sigma$ is as in hypothesis \eqref{I} and $L_r$ denotes the  H$\ddot{\text{o}}$lder constant of $r$ (see \eqref{Holderconstant}). Then,
	\begin{align*}
		W^{1, p(x)} \left(B_R\right) \subset L^{r^{R}} \left(B_R\right) \subset L^{r(x)} \left(B_R\right)
	\end{align*}
	is satisfied for all balls $B_{R} \Subset \Omega$ with radius $R \leq R_{0} $.
\end{lemma}
\begin{proof}
We consider a ball $B_R \Subset \Omega$ with radius $R \leq R_{0}$. Also, we take $x_1, x_2 \in \overline{B}_R$ such that $p\left(x_1\right) = p_{R}$ and $r\left(x_2\right) = r^{R}$. Then, according to \eqref{ROmegaprime} we have that
	\begin{align*}
		\frac{r^{R}}{p_{R}} & = \frac{r\left(x_1\right) + r\left(x_2\right) - r\left(x_1\right)}{p\left(x_1\right)}\\
		& \leq \frac{r\left(x_1\right)}{p\left(x_1\right)} + \frac{2^{\sigma} L_r R^{\sigma}}{p_0}\\
		& \leq \sup_{x \in \Omega}\frac{r(x)}{p(x)} + \frac{2^{\sigma} L_r R^{\sigma}_{0}}{p_0}\\
		& \leq \sup_{x \in \Omega}\frac{r(x)}{p(x)} + \frac{1}{2} \l(1 + \frac{\sigma}{N} - \sup_{x \in \Omega}\frac{r(x)}{p(x)} \r)\\
		& < 1 + \frac{\sigma}{N},
	\end{align*}
	since $\sup_{x \in \Omega}\frac{r(x)}{p(x)} < 1 + \frac{\sigma}{N} $ (see hypothesis \eqref{I}).
	Hence, as $\sigma \leq 1$ and $p(x) > 1$ for all $x \in \Omega$, we deduce that
	\begin{align*}
		r^{R} < p_{R} \,\frac{N + 1}{N} \leq \frac{N p_{R}}{N - p_{R}} = p_{R}^*.
	\end{align*}
	Therefore, we can use the classical Sobolev embedding results, and in this way we reach the claim.
\end{proof}

\begin{remark} \label{remarkT}
		Let $R_0$ be as given in Lemma \ref{Romegaprime}. We emphasize that if
	\begin{align} \label{condition-mu1mu2}
		\mu_1^{\frac{1}{q(x)}}, \ \mu_2^{\frac{1}{r(x)}} \in C^{0, \sigma_1}(\Omega), \quad \text{where }   \sigma_1 \in (0, 1],
	\end{align}
then we have   $\mu_1^{\frac{R_0}{q(x)}}, \ \mu_2^{\frac{R_0}{r(x)}} \in C^{0, \sigma_2}(\Omega)$ for some $\sigma_1 \leq \sigma_2 \leq 1$.
Further, taking into account that without any loss of generality we can suppose $\sigma \leq \sigma_2$, according to the proof of Lemma \ref{Romegaprime} we get
	\begin{align*}
		\frac{q^{R_{0}}}{p_{R_{0}}} < 1 + \frac{\sigma_2}{N} \quad \mbox{and} \quad \frac{r^{R_{0}}}{p_{R_{0}}} < 1 + \frac{\sigma_2}{N}.
	\end{align*}

\end{remark}

Now, we are ready to state our first regularity result. Precisely, we prove a
Caccioppoli-type inequality.

\begin{theorem} [Caccioppoli-type inequality] \label{cacioppoli}
	Let hypothesis \eqref{I} be satisfied and let $u \in W^{1, p(x)}(\Omega)$ be a local minimizer of the functional $\mathcal{F}_{\mathcal{T}}$ defined in \eqref{intfunctionalT}. Then there exists a positive constant $C$, depending by $C_{\Delta}$ (see Section \ref{sec3}), such that
	\begin{align} \label{R1R2}
		\int_{B_{R_1}} \mathcal{T}(x, |\nabla u|)\,  \diff x \leq C \int_{B_{R_2}} \mathcal{T}\l(x, \frac{\left\vert u - u_{B_{R_2}}\right\vert}{R_2 - R_1}\r)\,  \diff x
	\end{align}
and
\begin{align} \label{R1R2l}
	\int_{B_{R_1}} \mathcal{T}\left(x, \left\vert\nabla (u - l)_{\pm}\right\vert\right)\,  \diff x \leq C \int_{B_{R_2}} \mathcal{T}\l(x, \frac{(u - l)_{\pm}}{R_2 - R_1}\r)\,  \diff x
\end{align}
are valid for all concentric balls $B_{R_1} \subset B_{R_2} \Subset \Omega$, with $0 < R_1 < R_2 \leq R_0$, and all $l \in \mathbb{R}$, where $R_0$ is as given in Lemma \ref{Romegaprime} and $u_{B_{R_2}}$ is as in \eqref{media}.
\end{theorem}
\begin{proof}
	Let $t, s, R_1, R_2 \in \mathbb{R}$ be such that $0 < R_1 \leq t < s \leq R_2 \leq R_0$ where $R_0$ is as given in Lemma \ref{Romegaprime}. We consider a cut-off function $\theta \in C^{1}(\Omega)$ satisfying the following conditions:
	\begin{align} \label{theta}
0 \leq \theta \leq 1, \quad \theta \equiv 1 \ \text{on} \ B_{t}, 	\quad \theta \equiv 0 \ \text{outside} \ B_{s} \quad \text{and} \quad |\nabla \theta| \leq \frac{1}{s - t}
	\end{align}
and we in addition put $w : = u - \theta \, \left(u - u_{B_{R_2}}\right)$. We stress that according to \eqref{theta}, we have
\begin{align*}
w = u \quad \text{outside }  B_s  \text{ and further }  \text{supp} \,(u - w) = \text{supp} \,\theta\,\left(u - u_{B_{R_2}}\right) \subset \overline{B_s} \subset \Omega.
\end{align*}
Also, taking into account that $\nabla w = (1 - \theta) \nabla u - \left(u - u_{B_{R_2}}\right) \, \nabla \theta$, according to \eqref{additivity},
we get
\begin{align*}
\mathcal{T}(x, |\nabla w|) & = \mathcal{T}\left(x, |(1 - \theta) \nabla u - \left(u - u_{B_{R_2}}\right) \, \nabla \theta|\right)\\
& \leq C_{\Delta} \,\left[\mathcal{T}(x, (1 - \theta) |\nabla u|) + \mathcal{T}\left(x, \left\vert u - u_{B_{R_2}}\right\vert\nabla \theta|\right)\right]
\end{align*}
with $C_{\Delta} > 0$ depending by $r^+$ (see Section \ref*{sec3}).

Consequently, using the minimality of $u$ with respect to $w \in W^{1, p(x)}(\Omega)$ (see Definition \ref{minimizer}) along with $\theta \equiv 1 \ \text{on} \ B_{t}$ and according to \eqref{theta}, we get
\begin{align*}
	\int_{B_{t}} \mathcal{T}(x, |\nabla u|)\,  \diff x & \leq \int_{B_{s}} \mathcal{T}(x, |\nabla u|)\,  \diff x \leq \int_{B_{s}} \mathcal{T}(x, |\nabla w|)\,  \diff x \\
	& \leq C_{\Delta} \int_{B_{s}} \mathcal{T}(x, (1 - \theta) |\nabla u|) \, \diff x + C_{\Delta} \int_{B_{s}} \mathcal{T}\left(x, \left\vert u - u_{B_{R_2}}\right\vert\, |\nabla \theta|\right) \, \diff x\\
	& \leq C_{\Delta} \int_{B_{s} \setminus B_t} \mathcal{T}(x, |\nabla u|) \, \diff x +
	C_{\Delta} \int_{B_{s}} \mathcal{T}\l(x, \frac{|u - u_{B_{R_2}}|}{s - t}\r) \, \diff x.
\end{align*}
Now, we point out that all the integrals in the previous inequalities are finite, this happens according to the minimality of $u$ and to Lemma \ref{Romegaprime}. Therefore, we can apply the hole-filling method. Thus, we first add to both the sides $C_{\Delta} \,\int_{B_{t}} \mathcal{T}(x, |\nabla u|)\,  \diff x$ and then we divide them by $C_{\Delta} + 1$. In this way, we obtain the following estimate
\begin{align*}
\int_{B_{t}} \mathcal{T}(x, |\nabla u|)\,  \diff x  \leq& 	\frac{C_{\Delta}}{C_{\Delta} + 1} \int_{B_{s}} \mathcal{T}(x, |\nabla u|) \, \diff x \\
	& + \frac{C_{\Delta}}{C_{\Delta} + 1} \int_{B_{s}} \mathcal{T}\l(x, \frac{\left\vert u - u_{B_{R_2}}\right\vert}{s - t}\r) \, \diff x.
\end{align*}
Hence, we are in  position to use Lemma $4.2$ of Harjulehto-H\"{a}st\"{o}-Toivanen \cite{Harjulehto-Hasto-Toivanen-2017} which guarantees the existence of a positive constant $C$, depending by $C_{\Delta}$, such that
\begin{align*}
	\int_{B_{R_1}} \mathcal{T}(x, |\nabla u|)\,  \diff x \leq C \int_{B_{R_2}} \mathcal{T}\l(x, \frac{\left \vert u - u_{B_{R_2}}\right\vert}{R_2 - R_1}\r)\,  \diff x
\end{align*}
for all concentric balls $B_{R_1} \subset B_{R_2} \Subset \Omega$, with $0 < R_1 < R_2 \leq R_0$, that is, \eqref{R1R2} holds. Now, we underline that arguing in a similar way but choosing $w := u - \theta (u - l)_{\pm}$, with $l \in \R$, we can also deduce that \eqref{R1R2l} holds.
\end{proof}

Thanks to Lemma \ref{Romegaprime} and Theorem \ref{DFO}, we can now establish the following Sobolev-Poincar\'{e}-type inequality for the functional $\mathcal{F}_{\mathcal{T}}$.

\begin{theorem} [Sobolev-Poincar\'{e}-type inequality] \label{SobPonIne}
	Let hypothesis \eqref{I} and assumption \eqref{condition-mu1mu2}
	be satisfied. In addition, let $u \in W^{1, p(x)}(\Omega)$ be a function such that $ \mathcal{T}(x, |\nabla u|) \in L^1(\Omega)$. Then, there exist a positive constant $C$, depending on $N$, $p(\cdot)$, $q(\cdot)$, $r(\cdot)$, $\sigma_2$, $L_r$, and $0 < \delta < 1$, depending on $N$, $p(\cdot)$, $q(\cdot)$, $r(\cdot)$, $\sigma_2$, $L_r$, $d$, so that
\begin{align} \label{SobPoiEq}
	-\hspace{-0,38cm}	\int_{B_{R}} \mathcal{T}\l(x, \frac{\left\vert u - u_{B_{R}}\right\vert}{R}\r)\,  \diff x &
	\leq C + C \l(-\hspace{-0,38cm} \int_{B_R}  \mathcal{T}(x, |\nabla u|)^{\delta}\, \diff x \r)^{\frac{1}{\delta}}
\end{align}
	holds for any $B_R \Subset \Omega$ with $R \leq R_0$, where $R_0$ is as given in Lemma \ref{Romegaprime}.
\end{theorem}
\begin{proof} First, we stress that since $B_R \Subset \Omega$ we can find a ball $ B_{R_0} \Subset \Omega$, being $R_0$ as given in Lemma \ref{Romegaprime}, so that it results $B_R \subset B_{R_0}$. Then, according to the fact that $u \in W^{1, p(x)}(\Omega)$, we get $u \in W^{1, p(x)}(B_{R_0})$. Also, we recall that Lemma \ref{Romegaprime} guarantees that
	\begin{align*}
		W^{1, p(x)} \left(B_{R_0}\right) \subset L^{r^{R_0}} \left(B_{R_0}\right) \subset L^{r(x)} \left(B_{R_0}\right).
	\end{align*}
Consequently, we know that
\begin{align*}
	u \in L^{r^{R_0}} \left(B_{R_0}\right) \subset L^{r(x)} \left(B_{R_0}\right) \quad \text{and further} \quad
	u \in L^{q^{R_0}} \left(B_{R_0}\right) \subset L^{q(x)} \left(B_{R_0}\right).
\end{align*}
Now, as $B_R \subset B_{R_0}$ we have
	\begin{align*} \label{R1R2-}
 -\hspace{-0,38cm} \int_{B_{R}} \mathcal{T}\l(x, \frac{\left\vert u - u_{B_{R}}\right\vert}{R}\r)\,  \diff x
	 \leq& -\hspace{-0,38cm} \int_{B_{R}} \l(1 + \l|\frac{u - u_{B_{R}}}{R}\r|^{p^{R_0}}\r)\,  \diff x\\
	& + -\hspace{-0,38cm} \int_{B_{R}} \l(1 + \l(\mu_{1}^{\frac{1}{q(x)}} \l|\frac{u - u_{B_{R}}}{R}\r|\r)^{q^{R_0}}\r)\,  \diff x\\
	& + -\hspace{-0,38cm} \int_{B_{R}} \l(1 + \l(\mu_{2}^{\frac{1}{r(x)}} \l|\frac{u - u_{B_{R}}}{R}\r|\r)^{r^{R_0}}\r)\,  \diff x\\
	 \leq& C + C -\hspace{-0,38cm}\int_{B_R} \l|\frac{u - u_{B_{R}}}{R}\r|^{p^{R_0}}\,  \diff x\\
	& + C -\hspace{-0,38cm}\int_{B_R} \l(\mu_{1}^{\frac{1}{q(x)}} \l|\frac{u - u_{B_{R}}}{R}\r|\r)^{q^{R_0}} \, \diff x\\
& + C -\hspace{-0,38cm}\int_{B_R} \l(\mu_{2}^{\frac{1}{r(x)}} \l|\frac{u - u_{B_{R}}}{R}\r|\r)^{r^{R_0}} \, \diff x
\end{align*}
for some $C > 0$ depending on $N, p(\cdot), q(\cdot), r(\cdot)$. Then, according to Remark \ref{remarkT}, we can apply Theorem \ref{DFO} with $\tilde{p} = p^{R_0}, \ \tilde{q} = q^{R_0}, \ \tilde{r} = r^{R_0}$ and $\nu_1 = \mu_{1}^{q^{R_0}/q(x)}, \ \nu_2 = \mu_{2}^{r^{R_0}/r(x)}$ getting
\begin{align}
& -\hspace{-0,38cm} \int_{B_{R}} \mathcal{T}\l(x, \frac{\left\vert u - u_{B_{R}}\right\vert}{R}\r)\,  \diff x  \nonumber\\
\leq& C + C \l( -\hspace{-0,38cm}\int_{B_R} \l[|\nabla u|^{p^{R_0}} + \l(\mu_{1}(x)^{\frac{1}{q(x)}} |\nabla u |\r)^{q^{R_0}} + \l(\mu_{2}(x)^{\frac{1}{r(x)}} |\nabla u|\r)^{r^{R_0}}\r]^d \diff x \r)^{\frac{1}{d}} \nonumber\\
 \leq& C + C \, \l( -\hspace{-0,38cm}\int_{B_R}  |\nabla u |^{d \, p^{R_{0}}}  \, \diff x\r)^{\frac{1}{d}} \nonumber\\
	& + C  \, \l(-\hspace{-0,38cm}\int_{B_R} \l(\mu_1(x)^{\frac{1}{q(x)}} |\nabla u|\r)^{d q^{R_{0}}} \, \diff x \r)^{\frac{1}{d}}  \nonumber \\
	& + C \, \l(-\hspace{-0,38cm}\int_{B_R} \left(\mu_2(x)^{\frac{1}{r(x)}} |\nabla u|\right)^{d r^{R_{0}}} \, \diff x \r)^{\frac{1}{d}}
\end{align}
for some $C > 0$ depending on $N, p(\cdot), q(\cdot), r(\cdot)$ and $0 < d < 1$ depending on $N, p(\cdot), q(\cdot), r(\cdot), \sigma_2$.

Next, we stress that if we in addition choose $0 < R_0 \leq 1$ such that
\begin{align} \label{R0BIS}
&	R_{0}^\sigma < \frac{1-d}{d} \frac{p_0}{2^\sigma \max \left\{L_p, L_q, L_r\right\}},
\end{align}
then we have
\begin{align*}
	d \, p^{R_{0}} < p_{R_{0}}, \quad  d \, q^{R_{0}} < q_{R_{0}} \ \mbox{and} \quad d \, r^{R_{0}} < r_{R_{0}}.
\end{align*}
Therefore, we can find $\delta \in (d, 1)$ such that
\begin{equation}\label{Edtheta}
	d \, p^{R_{0}} < \delta \,p_{R_{0}}, \quad  d\, q^{R_{0}} < \delta \,q_{R_{0}} \ \mbox{and} \quad d \, r^{R_{0}} < \delta \,r_{R_{0}}.
\end{equation}
Further, using the H\"{o}lder inequality with respect to \eqref{Edtheta}, we can conclude the following estimate
\begin{align}\label{A2.19}
	\l( -\hspace{-0,38cm}\int_{B_R}  |\nabla u |^{d \, p^{R_{0}}} \, \diff x\r)^{\frac{1}{d}} \leq & \l(-\hspace{-0,38cm}\int_{B_R} |\nabla u|^{\delta \,p_{R_{0}}}\, \diff x \r)^{\frac{p^{R_{0}}}{\delta \,p_{R_{0}}}} \nonumber\\
	=& \left(-\hspace{-0,38cm}\int_{B_R} |\nabla u|^{\delta \,p_{R_{0}}} \, \diff x \right)^{\frac{p^{R_{0}} - p_{R_{0}}}{\delta \,p_{R_{0}}}} \l(-\hspace{-0,38cm}\int_{B_R} |\nabla u|^{\delta \,p_{R_{0}}} dx \r)^{\frac{1}{\delta}} \nonumber\\
	\leq & \l(-\hspace{-0,38cm}\int_{B_R}\left(1+ |\nabla u|^{ p(x)}\r)\, \diff x \right)^{\frac{ p^{R_{0}} - p_{R_{0}}}{\delta \,p_{R_{0}}}}  \nonumber\\
	&  \cdot \quad \l(-\hspace{-0,38cm}\int_{B_R}\left(1+ |\nabla u|^{\delta p(x)}\right) dx \r)^{\frac{1}{\delta}}.
\end{align}
Now, we recall that  $u \in W^{1, p(x)}(\Omega)$,
therefore we have that $\int_{B_{R_0}} |\nabla u|^{ p(x)}\, \diff x$ is bounded. Taking into account that $\left\vert B_{R_0}\right\vert^{-(p^{R_{0}} - p_{R_{0}})}$ is also bounded (see Lemma 4.1.6 of Diening-Harjulehto-H\"{a}st\"{o}-R$\mathring{\text{u}}$\v{z}i\v{c}ka \cite{Diening-Harjulehto-Hasto-Ruzicka-2011}), we can affirm that there exists a constant $e_1$ depending on $\mathcal{T}(x, | \nabla u |),$ $ p(\cdot),\, d, \, N, \delta$ so that
\begin{align*}
	\l(-\hspace{-0,38cm} \int_{B_{R}}  |\nabla u|^{p(x)} \, \diff
	 x \r)^{\frac{p^{R_{0}} - p_{R_{0}}}{\delta p_{R_{0}}}} \leq \l(-\hspace{-0,38cm} \int_{B_{R_0}}  |\nabla u|^{p(x)} \, \diff
	 x \r)^{\frac{p^{R_{0}} - p_{R_{0}}}{\delta p_{R_{0}}}} \leq e_1.
\end{align*}
On account of this, \eqref{A2.19} gives us the estimate
\begin{align} \label{est1}
	\l(-\hspace{-0,38cm}\int_{B_R} |\nabla u|^{d p^{R_{0}}}\, \diff x \r)^{\frac{1}{d}} \leq e_2 + e_2 \l(-\hspace{-0,38cm} \int_{B_R} |\nabla u|^{\delta p(x)} \, \diff x \r)^{\frac{1}{\delta}}
\end{align}
 for some positive constant $e_2$ depending by $\delta$ and $e_1$.

Similarly, we can also estimate the last two terms in  the right hand side of $(6.14)$. Precisely, using again the H\"{o}lder inequality with respect to \eqref{Edtheta}, we get that
\begin{align}
	&	\l(-\hspace{-0,38cm}\int_{B_R} \left(\mu_1(x)^{\frac{1}{q(x)}} |\nabla u|\right)^{d q^{R_{0}}} \, \diff x \r)^{\frac{1}{d}} \leq  \l(-\hspace{-0,38cm}\int_{B_R} \l(\mu_1(x)^{\frac{1}{q(x)}} |\nabla u|\r)^{\delta q_{R_{0}}} \, \diff x \r)^{\frac{q^{R_{0}}}{\delta  q_{R_{0}}}}\nonumber \\
	\leq & \l(-\hspace{-0,38cm}\int_{B_R}\l(\mu_1(x)^{\frac{1}{q(x)}} |\nabla u|\right)^{\delta q_{R_{0}}} \diff x \r)^{\frac{q^{R_{0}} - q_{R_{0}}}{\delta  q_{R_{0}}}} \l(-\hspace{-0,38cm}\int_{B_R} \left(\mu_1(x)^{\frac{1}{q(x)}} |\nabla u|\right)^{\delta q_{R_{0}}}dx \r)^{\frac{1}{\delta}}\nonumber \\
	\leq & \l(-\hspace{-0,38cm}\int_{B_R} \l(1+ \l(\mu_1(x)^{\frac{1}{q(x)}} |\nabla u|\r)^{q(x)}\r)\, \diff x \r)^{\frac{q^{R_{0}} - q_{R_{0}}}{\delta  q_{R_{0}}}} \nonumber \\
	&  \cdot \quad \l(-\hspace{-0,38cm} \int_{B_R} \l(1+ \l(\mu_1(x)^{\frac{1}{q(x)}} |\nabla u|\r)^{\delta q(x)}\r)\, \diff x \r)^{\frac{1}{\delta}}.
\end{align}
Next, $ \mathcal{T}(x, |\nabla u|) \in L^1(\Omega)$ and the fact that $\left\vert B_{R_0}\right\vert^{-(q^{R_{0}} - q{R_{0}})}$ is bounded assure that
\begin{align*}
& \l(	-\hspace{-0,38cm} \int_{B_R} \l(\mu_1(x)^{\frac{1}{q(x)}} |\nabla u| \r)^{q(x)}\, \diff x \r)^{\frac{q^{R_{0}} - q_{R_{0}}}{\delta q_{R_{0}}}}\\ & \leq \l(-\hspace{-0,38cm} \int_{B_{R_0}} \l(\mu_1(x)^{\frac{1}{q(x)}} |\nabla u| \r)^{q(x)}\, \diff x\r)^{\frac{p^{R_{0}} - p_{R_{0}}}{\delta p_{R_{0}}}}\\
& \leq e_3
\end{align*}
for some positive constant $e_3$ depending on $\mathcal{T}(x, | \nabla u |),\, q(\cdot),\, d, \, N, \delta$.
Based on this, we conclude that
\begin{align} \label{estnu1}
	\l(-\hspace{-0,38cm}\int_{B_R} \left(\mu_1(x)^{\frac{1}{q(x)}} |\nabla u|\right)^{d q^{R_{0}}} \, \diff x \r)^{\frac{1}{d}} \leq e_4 + e_4 \l(-\hspace{-0,38cm} \int_{B_R}  \l(\mu_1(x)^{\frac{1}{q(x)}} |\nabla u|\r)^{\delta q(x)}\, \diff x \r)^{\frac{1}{\delta}}
\end{align}
for some positive constant $e_4$ depending by $\delta$ and $e_3$.

Analogously, we also get that
\begin{align} \label{estnu2}
	\l(-\hspace{-0,38cm}\int_{B_R} \left(\mu_2(x)^{\frac{1}{r(x)}} |\nabla u|\right)^{d r^{R_{0}}} \, \diff x \r)^{\frac{1}{d}} \leq e_5 + e_5 \l(-\hspace{-0,38cm} \int_{B_R}  \l(\mu_2(x)^{\frac{1}{r(x)}} |\nabla u|\r)^{\delta r(x)}\, \diff x \r)^{\frac{1}{\delta}}
\end{align}
for some positive constant $e_5$ depending on $\mathcal{T}(x, | \nabla u |),\, r(\cdot),\, d, \, N, \delta$.

Now, using \eqref{est1}, \eqref{estnu1} and \eqref{estnu2} in $(6.14)$ we finally deduce that there exists a positive constant $C$, depending on $\mathcal{T}(x, | \nabla u |),$ $ p(\cdot), q(\cdot), r(\cdot), \, d, \, N, \delta$ and $\Omega$, so that
\begin{align*}
	-\hspace{-0,38cm} \int_{B_{R}} \mathcal{T}\l(x, \frac{|u - u_{B_{R}}|}{R}\r)\,  \diff x\leq C + C \l(-\hspace{-0,38cm} \int_{B_R}  \mathcal{T}(x, |\nabla u|)^{\delta}\, \diff x \r)^{\frac{1}{\delta}}
	\end{align*}
holds for all $B_R \Subset \Omega$ with $R \leq R_0$, that is, the claim is proved.
\end{proof}

We next emphasize that we can also obtain a Sobolev-Poincar\'{e}-type inequality for functions not necessarily zero on the boundary. In order to derive such a type of inequality we make use of Lemma \ref{Romegaprime} and Remark \ref{RemarkDeFilippisOh}.

\begin{theorem} \label{SobPoiInEE}
	Let hypothesis \eqref{I} be satisfied and let $E \subset B_R$ be a set with measure $|E|$ satisfying $|E| \geq \gamma |B_R|$ for some $0 < \gamma < 1$. In addition, let $u \in W^{1, p(x)}(\Omega)$ be a function
	such that
	\begin{align*}
		\mathcal{T}(x, |\nabla u|) \in L^1\left(B_R\right) \quad \text{and further} \quad u = 0 \ \text{a.\,e.\,on} \ E.
	\end{align*}
Then, there exist $0 < \delta < 1$ and a constant $C > 0$, independent of $u$ and $|E|$, so that
	\begin{align} \label{SobPoiIneE}
		-\hspace{-0,38cm}\int_{B_{R}} \mathcal{T}\l(x, \frac{ |u|}{R}\r)\,  \diff x \leq C \l[1 + \l(-\hspace{-0,38cm} \int_{B_R}  \mathcal{T}(x, |\nabla u|)^{\delta}\, \diff x \r)^{\frac{1}{\delta}} \r]
	\end{align}
	holds for all $B_R \Subset \Omega$ with $R \leq R_0$, where $0 < R_0 \leq 1$ satisfies \eqref{ROmegaprime} and \eqref{R0BIS}.
\end{theorem}
\begin{proof} We recall that since $B_R \Subset \Omega$ we can find $ B_{R_0} \Subset \Omega$, with $0 < R_0 \leq 1$ satisfying conditions \eqref{ROmegaprime} and \eqref{R0BIS}, so that it results $B_R \subset B_{R_0}$. In addition, from
	 Lemma \ref{Romegaprime} we know that
	\begin{align*}
		W^{1, p(x)} \left(B_{R_0}\right) \subset L^{r^{R_0}} \left(B_{R_0}\right)  \subset L^{r(x)} \left(B_{R_0}\right) .
	\end{align*}
	These facts along with $u \in W^{1, p(x)}(\Omega)$ allow us to deduce that
	\begin{align*}
		u \in L^{r^{R_0}} \left(B_{R_0}\right)  \subset L^{r(x)} \left(B_{R_0}\right)  \quad \text{and} \quad
		u \in L^{q^{R_0}} \left(B_{R_0}\right)  \subset L^{q(x)} \left(B_{R_0}\right).
	\end{align*}
	Hence, as $B_R \subset B_{R_0}$, similar to the proof of Theorem \ref{SobPonIne}, we can show that
	\begin{align*}
		-\hspace{-0,38cm} \int_{B_{R}} \mathcal{T}\l(x, \frac{|u|}{R}\r)\,  \diff x
		& \leq C + C  -\hspace{-0,38cm}\int_{B_R} \l|\frac{u}{R}\r|^{p^{R_0}}\,  \diff x\\
		& + C -\hspace{-0,38cm}\int_{B_R} \l(\mu_{1}^{\frac{1}{q(x)}} \l|\frac{u}{R}\r|\r)^{q^{R_0}} \, \diff x\\
		& + C  -\hspace{-0,38cm}\int_{B_R} \l(\mu_{2}^{\frac{1}{r(x)}} \l|\frac{u}{R}\r|\r)^{r^{R_0}} \, \diff x
	\end{align*}
	for some $C > 0$ depending on $N, p(\cdot), q(\cdot), r(\cdot)$. Then, using Remark \ref{RemarkDeFilippisOh} with $\tilde{p} = p^{R_0}, \ \tilde{q} = q^{R_0}$ and $\tilde{q} = q^{R_0}$, we get that
		\begin{align} \label{estimare}
		-\hspace{-0,38cm} \int_{B_{R}} \mathcal{T}\l(x, \frac{|u|}{R}\r)\,  \diff x
		& \leq C  \nonumber\\
		& + C \l( -\hspace{-0,38cm}\int_{B_R} \l[|\nabla u|^{p^{R_0}} + \l(\mu_{1}^{\frac{1}{q(x)}} |\nabla u |\r)^{q^{R_0}} + \l(\mu_{2}^{\frac{1}{r(x)}} |\nabla u|\r)^{r^{R_0}}\r]^d \diff x \r)^{\frac{1}{d}} \nonumber\\
		& \leq C  + C  \l(-\hspace{-0,38cm}\int_{B_R} |\nabla u|^{d \, p^{R_0}}\,  \diff x \r)^{\frac{1}{d}}\nonumber\\
		& + C  \l(-\hspace{-0,38cm}\int_{B_R} \l(\mu_{1}^{\frac{1}{q(x)}} |\nabla u|\r)^{d \, q^{R_0}} \, \diff x \r)^{\frac{1}{d}}\nonumber\\
		& + C  \l(-\hspace{-0,38cm}\int_{B_R} \l(\mu_{2}^{\frac{1}{r(x)}} |\nabla u|\r)^{d \, r^{R_0}} \, \diff x \r)^{\frac{1}{d}}.
	\end{align}
	Now, following again the proof of Theorem \ref{SobPonIne} we are able to derive a precise estimate for each of the addends in the right side of  \eqref{estimare}. Thus, using these estimates we conclude that \eqref{SobPoiIneE} is valid for all $B_R \Subset \Omega$ with $R \leq R_0$, where $R_0$ is as given in Lemma \ref{Romegaprime} and further it satisfies \eqref{R0BIS}.
\end{proof}

Now, the Caccioppoli-type inequality \eqref{R1R2} and the Sobolev-Poincar\'{e}-type inequality \eqref{SobPoiEq} allows us to prove the following result of higher integrability for the gradient of the minimizers of the functional $\mathcal{F}_{\mathcal{T}}$.

\begin{theorem} \label{HI}
	Let hypothesis \eqref{I} and assumption \eqref{condition-mu1mu2} be satisfied. Let $u \in W^{1, p(x)}(\Omega)$ be a local minimizer of the functional $\mathcal{F}_{\mathcal{T}}$ defined in \eqref{intfunctionalT}. Then, there exist positive constants $m_0$ and $C$, with $C$ depending on $N, p(\cdot), q(\cdot), r(\cdot), \sigma_2, L_r$, such that
	\begin{align*}
		\mathcal{T}(x, |\nabla u|) \in L^{1+ m_0}(\Omega)
	\end{align*}
	and in addition
	\begin{equation}
		\left(-\hspace{-0,36cm}\int_{B_{R/2}} \mathcal{T}(x, |\nabla u|)^{1+m_0} \, \diff x \right)^{\frac{1}{1+m_0}} \leq C + C -\hspace{-0,36cm}\int_{B_{R}} \mathcal{T}(x, |\nabla u|) \, \diff x	 \end{equation}
	holds for any $B_R \Subset \Omega$ with radius $R \leq R_0 $, where $0 < R_0 \leq 1$ satisfies \eqref{ROmegaprime} and \eqref{R0BIS}.
\end{theorem}
\begin{proof} According to Lemma \ref{Romegaprime}, we know that there exists $0 < R_{0} \leq 1$ such that for every ball $B_R \Subset \Omega$ with radius $R \leq R_{0}$ we have
	\begin{align*}
		W^{1, p(x)} \left(B_R\right) \subset L^{r^{R}} \left(B_R\right) \subset L^{r(x)} \left(B_R\right).
	\end{align*}
	From $u \in W^{1, p(x)}(\Omega)$, we deduce that $u \in W^{1, p(x)}\left(B_R\right)$. On this basis, for  given $B_R \Subset \Omega$ with radius $R \leq R_{0}$ we have that
	\begin{align*}
		u \in L^{r^{R}} \left(B_R\right) \subset L^{r(x)} \left(B_R\right)\quad \text{and further} \quad
		u \in L^{q^{R}} \left(B_R\right) \subset L^{q(x)} \left(B_R\right).
	\end{align*}
	Now, we can use the Caccioppoli-type inequality \eqref{R1R2} with $R_1 = R/2$ and $R_2 = R$ and thus we get that
	\begin{align*}
			\int_{B_{\frac{R}{2}}} \mathcal{T}(x, |\nabla u|)\,  \diff x & \leq C \int_{B_{R}} \mathcal{T}\l(x, \frac{\left\vert u - u_{B_{R}}\right\vert}{R}\r)\,  \diff x.
	\end{align*}
Taking into account that the Sobolev-Poincar\'{e}-type inequality \eqref{SobPoiEq} guarantees that the estimate
\begin{align*}
	-\hspace{-0,38cm} \int_{B_{R}} \mathcal{T}\l(x, \frac{\left\vert u - u_{B_{R}}\right\vert}{R}\r)\,  \diff x &
	\leq C + C \l(-\hspace{-0,38cm} \int_{B_R}  \mathcal{T}(x, |\nabla u|)^{\delta}\, \diff x \r)^{\frac{1}{\delta}}
\end{align*}
holds for all $B_R \Subset \Omega$ with $R \leq R_0$ and $R_0$ satisfying \eqref{R0BIS}, then we have also that
\begin{align*}
	-\hspace{-0,38cm}\int_{B_{\frac{R}{2}}} \mathcal{T}(x, |\nabla u|)\,  \diff x \leq C + C \l(-\hspace{-0,38cm} \int_{B_R}  \mathcal{T}(x, |\nabla u|)^{\delta}\, \diff x \r)^{\frac{1}{\delta}}
\end{align*}
holds for all $B_R \Subset \Omega$ with $R \leq R_0 $ where $R_0$ is as given in Lemma \ref{Romegaprime} and further it satisfies \eqref{R0BIS}. So, we can apply the well-known reverse H\"{o}lder inequality with increasing domain, due to Giaquinta-Modica (see Theorem \ref{Giaquinata}), getting the claim.

\end{proof}

We conclude this section by giving a higher integrability result up to the boundary for the gradient of the minimizers of the functional $\mathcal{F}_{\mathcal{T}}$. We point out that in order to establish this result we use Sobolev-Poincar\'{e}-type inequality \eqref{SobPoiIneE} along with inequalities \eqref{R1R2} and \eqref{SobPoiEq}.

\begin{theorem} \label{HI2}
	Let hypothesis \eqref{I} and assumption \eqref{condition-mu1mu2} be satisfied. Let $B_{2R} \Subset \Omega$
	and let $w \in W^{1, p(x)}(\Omega)$ be a function such that
	\begin{align*}
		\int_{B_{2 R}} \mathcal{T} (x, |\nabla w|)^{1 + m} \, \diff x < + \infty
	\end{align*}
	for some $m > 0$. In addition, we assume that $R \leq R_0$ where $0 < R_0 \leq 1$ satisfies \eqref{ROmegaprime} and \eqref{R0BIS}. If $v$ is a minimizer of the functional $\mathcal{F}_{\mathcal{T}} (u; B_R)$ such that $v = w$ on $\partial B_R$, then there exist $0 < m_1 < m$ and a positive constant $C$, depending on $p(\cdot), q(\cdot), r(\cdot), N$, such that
	\begin{align*}
		\mathcal{T}(x, |\nabla v|)^{1 + m_1} \in L^1\left(B_R\right),
	\end{align*}
	and further
	\begin{align*}
		\left(-\hspace{-0,36cm}\int_{B_{R}} \mathcal{T}(x, |\nabla v|)^{1+m_1} \, \diff x \right)  \leq& C 	\left(-\hspace{-0,36cm}\int_{B_{2 R}} \mathcal{T}(x, |\nabla v|) \, \diff x \right)^{1+m_1} \\
		& + C 	\left(-\hspace{-0,36cm}\int_{B_{2 R}} \mathcal{T}(x, |\nabla w|)^{1+m_1} \, \diff x \right) + C
		\end{align*}
	holds true.
\end{theorem}
\begin{proof}
At first, we consider $B_R \Subset \Omega$ and we choose $y \in B_R$ and $\varrho > 0$ so that we have $B_{2 \varrho}(y) \subset B_R$. Then, as $R \leq R_0$, using the Caccioppoli-type inequality \eqref{R1R2} with $R_1 = \varrho$ and $R_2 = 2 \varrho$ we get
\begin{align*}
		\int_{B_{\varrho}(y)} \mathcal{T}(x, |\nabla v|)\,  \diff x & \leq C \int_{B_{2 \varrho}(y)} \mathcal{T}\l(x, \frac{|v - v_{2 \varrho}|}{\varrho}\r)\,  \diff x.
\end{align*}
Now, the Sobolev-Poincar\'{e}-type inequality \eqref{SobPoiEq} assures that
\begin{align*}
	-\hspace{-0,38cm} \int_{B_{2 \varrho}(y)} \mathcal{T}\l(x, \frac{|v - v_{2 \varrho}|}{\varrho}\r)\,  \diff x &
	\leq C + C \l(-\hspace{-0,38cm} \int_{B_{2 \varrho}(y)}  \mathcal{T}(x, |\nabla v|)^{\delta}\, \diff x \r)^{\frac{1}{\delta}}.
\end{align*}
 On account of this, we derive that
\begin{align*}
	-\hspace{-0,38cm}\int_{B_{\varrho}(y)} \mathcal{T}(x, |\nabla v|)\,  \diff x \leq C + C \l(-\hspace{-0,38cm} \int_{B_{2 \varrho}(y)}  \mathcal{T}(x, |\nabla v|)^{\delta}\, \diff x \r)^{\frac{1}{\delta}}.
\end{align*}
 Therefore, we can apply Theorem \ref{Giaquinata} which allows us to obtain the higher integrability for the gradient of the minimizers of $\mathcal{F}_{\mathcal{T}}$.

Let now $y \in \partial B_R$ and $\varrho > 0$ be such that $B_{2 \varrho}(y) \subset B_{2 R} \Subset \Omega$. It follows that $2 \rho \leq R$. Also, we take $t, s \in \mathbb{R}$ so that $\varrho \leq t < s \leq 2 \varrho$ and let $u$ be the function defined by
\begin{align} \label{u}
	u(x)=
	\begin{cases}
		v(x) & \text{if } x \in B_R,\\
		w(x)& \text{if } x \in B_{2 R} \setminus B_R.
	\end{cases}
\end{align}
Next, we consider a cut-off function $\theta_1 \in C^1(\Omega)$ satisfying the following conditions:
\begin{align} \label{theta2}
	0 \leq \theta_1 \leq 1, \quad \theta_1 \equiv 1 \ \text{on} \ B_{t}(y), 	\quad \theta_1 \equiv 0 \ \text{outside} \ B_{s}(y) \quad \text{and} \quad |\nabla \theta_1| \leq \frac{1}{s - t}
\end{align}
and we in addition put $h : = v - \theta_1 \, (v - w)$. Taking into account that
\begin{align*}
	|\nabla h| \leq \left(1 - \theta_1\right) |\nabla v| + \theta_1 |\nabla w| + \frac{|w - v|}{s - t},
\end{align*}
then we have
 \begin{align*}
 	\int_{B_{t}(y) \cap B_R} \mathcal{T}(x, |\nabla v|)\,  \diff x  \leq& \int_{B_{s}(y) \cap B_R} \mathcal{T}(x, |\nabla v|)\,  \diff x\\
 	  \leq& \int_{B_{s}(y) \cap B_R} \mathcal{T}(x, |\nabla h|)\,  \diff x \\
 	 \leq& C \int_{(B_{s}(y) \setminus B_t(y)) \cap B_R} \mathcal{T}(x, |\nabla v|) \, \diff x \\
 	& + C \int_{B_{s}(y) \cap B_R} \mathcal{T}\l(x, \frac{|w - v|}{s - t}\r) \, \diff x \\
 	& + C \int_{B_{s}(y) \cap B_R} \mathcal{T}\l(x, |\nabla w|\r) \, \diff x.
 \end{align*}
Now, we apply the hole-filling method. Thus, we add to both the sides
\begin{align*}
	C \,\int_{B_{t}(y) \cap B_R} \mathcal{T}(x, |\nabla v|)\,  \diff x
\end{align*} and then we divide the obtained inequality by $C + 1$. In this way, we derive that
\begin{align*}
	\int_{B_{t} (y) \cap B_R} \mathcal{T}(x, |\nabla v|)\,  \diff x  \leq& 	\frac{C}{C + 1} \int_{B_{s}(y) \cap B_R} \mathcal{T}(x, |\nabla v|) \, \diff x \\
	& + \frac{C}{C + 1} \int_{B_{s}(y) \cap B_R} \mathcal{T}\l(x, \frac{|w - v|}{s - t}\r) \, \diff x \\
	& + \frac{C}{C + 1} \int_{B_{s}(y) \cap B_R} \mathcal{T}\l(x, |\nabla w|\r) \, \diff x.
\end{align*}
At this point, thanks to a standard iterative method (see Harjulehto-H\"{a}st\"{o}-Toivanen \cite[Lemma 4.2]{Harjulehto-Hasto-Toivanen-2017}), we obtain that
\begin{align*}
	\int_{B_{\varrho} (y) \cap B_R} \mathcal{T}(x, |\nabla v|)\,  \diff x  \leq& C \int_{B_{2 \varrho} (y) \cap B_R} \mathcal{T}\l(x, \frac{| w - v|}{\varrho}\r)\,  \diff x+ C \int_{B_{2 \varrho} (y) \cap B_R} \mathcal{T}(x, |\nabla w|)\,  \diff x\\
	 =&  C \int_{B_{2 \varrho} (y)} \mathcal{T}\l(x, \frac{| w - u|}{\varrho}\r)\,  \diff x+ C \int_{B_{2 \varrho} (y) \cap B_R} \mathcal{T}(x, |\nabla w|)\,  \diff x
\end{align*}
(we remark that $ w - v = w - u$ on $B_R$ and $w - u = 0$ on $B_{2 \varrho} (y) \setminus B_R$ (see \eqref{u})).

According of this and with a view to \eqref{u}, we have
\begin{align*}
	\int_{B_{\varrho} (y)} \mathcal{T}(x, |\nabla u|)\,  \diff x  =& \int_{B_{\varrho} (y) \cap B_R} \mathcal{T}(x, |\nabla v|)\,  \diff x+ \int_{B_{\varrho} (y) \setminus B_R} \mathcal{T}(x, |\nabla w|)\,  \diff x\\
	 \leq& C  \int_{B_{2 \varrho} (y) } \mathcal{T}\l(x, \frac{| w - u|}{\varrho}\r)\,  \diff x+ C \int_{B_{2 \varrho} (y)} \mathcal{T}(x, |\nabla w|)\,  \diff x.
\end{align*}
Since $2 \rho \leq R \leq R_0$, we can use the Sobolev-Poincar\'{e}-type inequality \eqref{SobPoiIneE} with $E = B_{2 \varrho}(y) \setminus B_R$ to obtain
\begin{align*}
		-\hspace{-0,38cm} \int_{B_{\varrho} (y)} \mathcal{T}(x, |\nabla u|)\,  \diff x  \leq&
		C \l(-\hspace{-0,38cm} \int_{B_{2 \varrho}(y)}  \mathcal{T}(x, |\nabla (w - u)|)^{\delta}\, \diff x \r)^{\frac{1}{\delta}}+ C  \int_{B_{2 \varrho} (y)} \mathcal{T}(x, |\nabla w|)\,  \diff x + C
\end{align*}
for some $0 < \delta < 1$. Hence, we get
\begin{align*}
	-\hspace{-0,38cm} \int_{B_{\varrho} (y)} \mathcal{T}(x, |\nabla u|)\,  \diff x  \leq&
	C \l(-\hspace{-0,38cm} \int_{B_{2 \varrho}(y)}  \mathcal{T}(x, |\nabla  u|)^{\delta}\, \diff x \r)^{\frac{1}{\delta}}+ C -\hspace{-0,38cm} \int_{B_{2 \varrho} (y)} \mathcal{T}(x, |\nabla w|)\,  \diff x + C.
\end{align*}
Finally, thanks to Theorem \ref{Giaquinata} we know that there exists $0 < m_1 < \delta$ so that the following estimate
\begin{align*}
	-\hspace{-0,38cm} \int_{B_{\varrho} (y)} \mathcal{T}(x, |\nabla u|)^{1 + m_1}\,  \diff x
	 \leq& C \l(-\hspace{-0,38cm} \int_{B_{2 \varrho}(y)}  \mathcal{T}(x, |\nabla  u|)\, \diff x \r)^{1 + m_1}\\
	& + C -\hspace{-0,38cm} \int_{B_{2 \varrho} (y)} \mathcal{T}(x, |\nabla w|)^{1 + m_1}\,  \diff x + C
\end{align*}
holds true. Hence, we deduce that $\mathcal{T}(x, |\nabla u|)^{1 + m_1} \in L^1\left(B_{\varrho}(y)\right)$, and consequently we have  $\mathcal{T}(x, |\nabla v|)^{1 + m_1} \in L^1(B_{\varrho}(y) \cap B_R)$.
Now, the claim follows by a simple covering argument.
\end{proof}

\section*{Statements \& Declarations}

\noindent \textbf{Competing interests.}  The authors read and approved the final manuscript. The authors have no relevant financial or non-financial interests to disclose.\\

\noindent \textbf{Availability of data and material.} This paper has no associated data and material.



\begin{thebibliography}{99}

\bibitem{Aberqi-Bennouna-Benslimane-Ragusa-2022}
	A. Aberqi et al.,
	{\it Existence results for double phase problem in Sobolev-Orlicz spaces with variable exponents in complete manifold},
	Mediterr. J. Math. {\bf 19} (2022), no. 4, Paper No. 158, 19 pp.

\bibitem{Albalawi-Alharthi-Vetro-2022}
	K.S. Albalawi,	N.H. Alharthi and F. Vetro,
	{\it Gradient and parameter dependent Dirichlet $(p(x), q(x))$-Laplace type problem},
	Mathematics  {\bf 10} (2022), no. 8, Article number 1336, 15 pp.

\bibitem{Bahrouni-Radulescu-Winkert-2020}
	A. Bahrouni, V.D. R\u{a}dulescu and P. Winkert,
	{\it Double phase problems with variable growth and convection for the Baouendi-Grushin operator},
	Z. Angew. Math. Phys. {\bf 71} (2020), no. 6, 183, 14 pp.

\bibitem{Baroni-Colombo-Mingione-2018}
	P. Baroni, M. Colombo and G. Mingione,
	{\it Regularity for general functionals with double phase},
	Calc. Var. Partial Differential Equations {\bf 57} (2018), no. 2, Art. 62, 48 pp.

\bibitem{Bauer-2001}
H. Bauer, {\it Measure and integration theory}, De Gruyter, Berlin, 2001.

\bibitem{Bogachev-2007}
V.I. Bogachev, {\it Measure theory}, Springer-Verlag, Berlin, 2007.

\bibitem{Colasuonno-Squassina-2016}
	F. Colasuonno and M. Squassina,
	{\it Eigenvalues for double phase variational integrals},
	Ann. Mat. Pura Appl. (4) {\bf 195} (2016), no. 6, 1917--1959.
	
	\bibitem{Colombo-Mingione-2015}
	M. Colombo and G. Mingione,
	{\it Regularity for double phase variational problems}, Arch. Rat. Mech. Anal. {\bf 215} (2015),
	443--496, doi:10.1007/s00205-014-0785-2.
	
	\bibitem{Colombo-Mingione-2015-1}
	 M. Colombo and G. Mingione,
	{\it Bounded minimizers of double phase variational integrals}, Arch. Rat. Mech. Anal. {\bf 218}
	(2015), 219--273, doi:10.1007/s00205-015-0859-9.

\bibitem{Crespo-Blanco-Gasinski-Harjulehto-Winkert-2022}
	\'{A}. Crespo-Blanco et al.,
	{\it A new class of double phase variable exponent problems: Existence and uniqueness},
	J. Differential Equations {\bf 323} (2022), 182--228.

	
	\bibitem{DeFilippis-Oh-2019}
	 C. De Filippis and J. Oh,
	{\it Regularity for multi-phase variational problems}, J. Differential Equations, {\bf 267} (2019), 1631--1670.

\bibitem{Diening-Harjulehto-Hasto-Ruzicka-2011}
	L. Diening, P. Harjulehto and P. H\"{a}st\"{o}, M. R$\mathring{\text{u}}$\v{z}i\v{c}ka,
	{\it Lebesgue and Sobolev spaces with variable exponents},
	Springer, Heidelberg, 2011.

\bibitem{Fan-2010}
X. Fan, {\it Sobolev embeddings for unbounded domain with variable exponent having values across $N$}, Math. Inequal.
Appl. {\bf 13} (2010), no. 1, 123--134.

\bibitem{Fan-2012}
X. Fan, {\it Differential equations of divergence form in Musielak-Sobolev spaces and a sub-supersolution method}, J.
Math. Anal. Appl. {\bf 386} (2012). no. 2, 593--604.

\bibitem{Fan-Guan-2010}
X. Fan, C.-X. Guan, {\it Uniform convexity of Musielak-Orlicz-Sobolev spaces and applications}, Nonlinear Anal. {\bf 73}
(2010), no. 1, 163--175.

\bibitem{Gasinski-Papageorgiou-2006}
L. Gasi\'nski and N.S. Papageorgiou,
{\it Nonlinear Analysis}, Chapman and Hall/CRC, Boca Raton,
FL, 2006.

\bibitem{Gasinski-Papageorgiou-2021}
	L. Gasi\'nski and N.S. Papageorgiou,
	{\it Constant sign and nodal solutions for superlinear double phase problems},
	Adv. Calc. Var. {\bf 14} (2021), no. 4, 613--626.

\bibitem{Gasinski-Winkert-2020}
	L. Gasi\'{n}ski and P. Winkert,
	{\it Constant sign solutions for double phase problems with superlinear nonlinearity},
	Nonlinear Anal. {\bf 195} (2020), 111739.

\bibitem{Gasinski-Winkert-2020b}
	L. Gasi\'nski and P. Winkert,
	{\it Existence and uniqueness results for double phase problems with convection term},
	J. Differential Equations {\bf 268} (2020), no. 8, 4183--4193.

\bibitem{Giusti-2003}
E. Giusti,
	{\it Direct methods in the calculus of variations}, World Scientific Publishing, River Edge, NJ, (2003).

\bibitem{Harjulehto-Hasto-2019}
	P. Harjulehto and P. H\"{a}st\"{o},
	{\it Orlicz spaces and generalized {O}rlicz spaces},
	Springer, Cham, 2019.
	
\bibitem{Harjulehto-Hasto-Toivanen-2017}
	P. Harjulehto, P. H\"{a}st\"{o} and O. Toivanen,
	{\it H\"{o}lder regularity of quasiminimizers under general growth conditions}, Calc. Var. Partial Differential Equations, {\bf 56} (2017), 22.

\bibitem{Le-2006}
	A. L\^{e},
	{\it Eigenvalue problems for the $p$-Laplacian},
	Nonlinear Anal. {\bf 64} (2006), no. 5, 1057--1099.

\bibitem{Leonardoi-Papageorgiou-2022}
	S. Leonardi and N.S. Papageorgiou,
	{\it Anisotropic Dirichlet double phase problems with competing nonlinearities},
	Rev. Mat. Complut., https://doi.org/10.1007/s13163-022-00432-3.

    \bibitem{Liu-Dai-2018}
	W. Liu and G. Dai,
	{\it Existence and multiplicity results for double phase problem}, J. Differential Equations {\bf 265} (2018), no. 9, 4311--4334.

	\bibitem{Marcellini-1989}
	P. Marcellini,
	{\it Regularity of minimizers of integrals of the calculus of variations with nonstandard growth conditions}, Arch. Rational Mech. Anal. {\bf 105} (1989), no. 3, 267--284.
	
	\bibitem{Marcellini-1991}
	P. Marcellini,
	{\it Regularity and existence of solutions of elliptic equations with $p,q$-growth conditions},
	J. Differential Equations {\bf 90} (1991), no. 1, 1--30.
	
	\bibitem{Motreanu-Motreanu-Papageorgiou-2014}
	D. Motreanu, V. Motreanu and N.S. Papageorgiou,
	{\it Topological and variational methods with applications to nonlinear boundary value problems},
	Springer, New York, 2014.
	
	\bibitem{Musielak-1983}
	J. Musielak, {\it Orlicz spaces and modular spaces}, Springer-Verlag, Berlin, 1983.
	
	\bibitem{Papageorgiou-Radulescu-Repovs-2019-2}
	N. S. Papageorgiou, V.D. R\u{a}dulescu and D.D. Repov\v{s},
	{\it Double-phase problems and a discontinuity property of the spectrum},
	Proc. Amer. Math. Soc. {\bf 147} (2019), 2899--2910.

\bibitem{Papageorgiou-Vetro-2019}
	N.S. Papageorgiou and C. Vetro,
	{\it Superlinear {$(p(z),q(z))$}-equations},
	Complex Var. Elliptic Equ. {\bf 64} (2019), no. 1, 8--25.

\bibitem{Papageorgiou-Vetro-Vetro-2021}
	N.S. Papageorgiou, C. Vetro and F. Vetro,
	{\it Solutions for parametric double phase Robin problems},
	Asymptot. Anal. {\bf 121} (2021), no. 2, 159--170.

\bibitem{Papageorgiou-Winkert-2018}
	N.S. Papageorgiou and P. Winkert,
	{\it Applied nonlinear functional analysis},
	De Gruyter, Berlin, 2018.

\bibitem{Perera-Squassina-2018}
	K. Perera and M. Squassina,
	{\it Existence results for double-phase problems via Morse theory},
	Commun. Contemp. Math. {\bf 20} (2018), no. 2, 1750023, 14 pp.
	
	\bibitem{Ragusa-Tachikawa-2020}
M.A. Ragusa and A. Tachikawa,
	{\it Regularity for minimizers for functionals of
		double phase with variable exponents},
Adv. Nonlinear Anal. {\bf 9} (2020), 710--728.


\bibitem{Vetro-Winkert-2022}
	F. Vetro and P. Winkert,
	{\it Existence, uniqueness and asymptotic behavior of parametric aniso\-tropic $(p,q)$-equations with convection},
	Appl. Math. Optim. {\bf 86} (2022), no. 2, Paper No. 18, 18 pp.

\bibitem{Vetro-Winkert-2023}
	F. Vetro and P. Winkert,
	{\it Constant sign solutions for double phase problems with variable exponents},
	Appl. Math. Lett. {\bf 135} (2023), Paper No. 108404, 7 pp.

\bibitem{Zeng-Radulescu-Winkert-2022}
	S. Zeng, V.D. R\u{a}dulescu and P. Winkert,
	{\it Double phase obstacle problems with variable exponent},
	Adv. Differential Equations {\bf 27} (2022), no. 9-10, 611--645.
	
\bibitem{Zhang-Radulescu-2018}
	Q. Zhang and V.D. R\u{a}dulescu,
		{\it Double phase anisotropic variational problems and combined effects of reaction and absorption terms},
	J. Math. Pures Appl. {\bf 118} (2018), 159--203.

\bibitem{Zhikov-1986}
	V.V. Zhikov,
	{\it Averaging of functionals of the calculus of variations and elasticity theory},
	Izv. Akad. Nauk SSSR Ser. Mat. {\bf 50} (1986), no. 4, 675--710.

\bibitem{Zhikov-1995}
	V.V. Zhikov,
	{\it On Lavrentiev's phenomenon},
	Russian J. Math. Phys. {\bf 3} (1995), no. 2, 249--269.

\bibitem{Zhikov-2011}
	V.V. Zhikov,
	{\it On variational problems and nonlinear elliptic equations with nonstandard growth conditions},
	J. Math. Sci. {\bf 173} (2011), no. 5, 463--570.

\end{thebibliography}
\end{document}